\documentclass[11pt]{amsart}
\usepackage{fullpage}
\usepackage[utf8]{inputenc} 
\usepackage[T1]{fontenc} 
\usepackage{mathtools} 
\usepackage{amsmath,amssymb,amsthm} 
\usepackage{dsfont} 
\usepackage{graphicx} 
\usepackage{float}
\usepackage{grffile}
\usepackage{pgf,tikz} 
\usepackage[all]{xy}
\usepackage{mathrsfs} 
\usepackage{textcomp} 
\usepackage{xcolor}
\usepackage{enumitem}
\usepackage[titletoc]{appendix} 
\usepackage{url}
\usepackage{esint} 
\usepackage[]{longtable}

\usepackage[pdfencoding=auto, psdextra]{hyperref}
\hypersetup{
    plainpages=false,
    unicode=true,          
    pdftoolbar=true,        
    pdfmenubar=true,        
    pdffitwindow=true,      
    pdftitle={My title},    
    pdfauthor={Author},     
    pdfsubject={Subject},   
    pdfcreator={Creator},   
    pdfproducer={Producer}, 
    pdfkeywords={keywords}, 
    pdfnewwindow=true,      
    colorlinks=true,       
    linkcolor=blue,          
    citecolor=blue,        
    filecolor=blue,      
    urlcolor=magenta,          
    urlbordercolor=0 1 1
}

\theoremstyle{plain}
\newtheorem{thm}{Theorem}[section]
\newtheorem{prop}[thm]{Proposition}

\newtheorem{lem}[thm]{Lemma}
\newtheorem{conj}[thm]{Conjecture}
\newtheorem*{thm*}{Theorem}
\newtheorem{Def}[thm]{Definition}

\theoremstyle{remark}
\newtheorem{rmk}[thm]{Remark}

\usepackage{palatino}
\allowdisplaybreaks
\numberwithin{equation}{section}

\newcommand{\bN}{\mathbb{N}}

\newcommand{\bR}{\mathbb{R}}
\newcommand{\bS}{\mathbb{S}}

\newcommand{\cD}{\mathcal{D}}

\newcommand{\cH}{\mathcal{H}}

\newcommand{\cL}{\mathcal{L}}

\newcommand{\cP}{{\mathcal P}}

\newcommand{\rar}{\rightarrow}

\newcommand{\SBV}{\mathrm{SBV}}

\def\ds{\displaystyle}

\def\mdet{\operatorname{det}}
\def\dist{\operatorname{dist}}
\newcommand{\res}{\mathop{\hbox{\vrule height 7pt width .5pt depth 0pt
\vrule height .5pt width 6pt depth 0pt}}\nolimits}

\title{A penalized Allen-Cahn equation for the mean curvature flow of thin structures}
\author{Elie Bretin}
\address{Univ Lyon, INSA de Lyon, CNRS UMR 5208, Institut Camille Jordan\\ 20 avenue Albert Einstein, F-69621 Villeurbanne, France\\ elie.bretin@insa-lyon.fr}
 \author{Chih-Kang Huang}
 \address{Univ Lyon, Universit\'e Claude Bernard Lyon 1, CNRS UMR 5208, Institut Camille Jordan \\43 boulevard du 11 novembre
 	1918, F-69622 Villeurbanne, France\\chih-kang.huang@hotmail.com}
 \author{Simon Masnou}
 \address{Univ Lyon, Universit\'e Claude Bernard Lyon 1, CNRS UMR 5208, Institut Camille Jordan \\43 boulevard du 11 novembre
 	1918, F-69622 Villeurbanne, France\\masnou@math.univ-lyon1.fr}

\makeatletter
\@namedef{subjclassname@2020}{\textup{2020} Mathematics Subject Classification}
\makeatother

\date{\today}

\begin{document} 

\date{\today}
\maketitle

\begin{abstract}
This paper addresses the approximation of the mean curvature flow of thin structures for which classical phase field methods are not suitable. By thin structures, we mean surfaces that are not domain boundaries, typically higher codimension objects such as 1D curves in 3D, i.e. filaments, or soap films spanning a boundary curve. To approximate the mean curvature flow of such surfaces, we consider a small thickening and we apply to the thickened set an evolution model that combines the classical Allen-Cahn equation with a penalty term that takes on larger values around the skeleton of the set. The novelty of our approach lies in the definition of this penalty term that guarantees a minimal thickness of the evolving set and prevents it from disappearing unexpectedly. We prove a few theoretical properties of our model, provide examples showing the connection with higher codimension mean curvature flow, and introduce a quasi-static numerical scheme with explicit integration of the penalty term. We illustrate the numerical efficiency of the model with accurate approximations of filament structures evolving by mean curvature flow, and we also illustrate its ability to find complex 3D approximations of solutions to the Steiner problem or the Plateau problem.
\end{abstract}
\maketitle

\section{Introduction}

The aim of this paper is to introduce a model and a numerical method for approximating the evolution of thin structures via mean curvature flow. As a specific example of thin structures, we focus on filaments in 3D space, which are typical codimension 2 objects possibly with singularities.  The method we introduce is actually suitable for more general structures such as surfaces which are not domain boundaries. It is therefore relevant to various physical applications, as we illustrate with numerical approximations of solutions to the Steiner or Plateau problems in 3D. 

The evolution by mean curvature of thin structures such as filaments falls into the category of higher codimension mean curvature flows, i.e. mean curvature flows of submanifolds with codimension at least 2.  A first list of references (see~\cite{BellettiniBook}) for the definition, the analysis and the approximation of higher-codimension flows is~\cite{altschuler1991singularities, degiorgi, ambrosio1996level, BelNov99, slepcev, wang2002long,ambrosio1997measure,BOS06}. More references and details are provided in the survey paper~\cite{smoczyk2011mean}.

Regarding the numerical approximation of higher codimension mean curvature flows, the first schemes capable of handling multiple junctions and interfaces were introduced in \cite{merriman1994motion} using a level set approach, see also~\cite{laux2019analysis}. The curve shortening flow in higher codimension was approximated numerically with finite elements in \cite{dziuk,deckelnik-dziuk,barrett2010,barrett2012,DN19}, with a semi-Lagrangian scheme and a level set approach in \cite{carlini} (for codimension 2 curves), and with a parametric approach in \cite{MU-14}. Finite elements were also used for approximating the anisotropic curve shortening \cite{Pozzi07} or the mean curvature flow of surfaces in higher or even arbitrary codimension \cite{Pozzi08,binz2023convergent}. 

In general, the above numerical approaches are not as effective as their counterparts designed for codimension 1 surfaces. This paper aims to address this limitation by showing how classical phase field methods can be used to approximate, at least formally, the mean curvature flow of codimension 2 interfaces (possibly combined with boundary constraints). This is achieved considering tubular neighborhoods of the interfaces and their mean curvature flow associated with thickness constraints.

\subsection{Classical mean curvature flow and its phase field approximation}

The classical mean curvature flow starting from a smooth open set $\Omega \subset \bR^N$ is a time-dependent set $\Omega(t)$ whose inner normal velocity $V_n(x)$, as long as it can be defined, at each point $x$ of $\partial\Omega(t)$ is proportional to the scalar mean curvature $H_{\partial\Omega(t)}(x)$. Here, we use the sign convention that the scalar mean curvature of a convex surface is nonnegative.
Up to time rescaling, the evolution flow takes the following form
$$
V_n(x) =H_{\partial \Omega(t)}(x), \quad x \in \partial \Omega(t),
$$
which corresponds to the evolution of $\partial \Omega(t)$ by the $L^2$-gradient flow of the perimeter 
$$
P(\Omega(t)) = \cH^{N-1}(\partial \Omega(t)),
$$
where $\cH^{N-1}$ is the $(N-1)$-Hausdorff measure.

 This emblematic geometric flow is the subject of a rich literature, see for example~\cite{amb-mant-96, BellettiniBook} and references therein. To deal with singularities, which can develop in finite time and mark the limit of existence of the smooth flow, various notions of weak mean curvature flows have been introduced, see the references in~\cite{amb-mant-96, BellettiniBook}. The numerical approximation of mean curvature flow has been addressed with various models and numerical methods, in particular parametric methods, level set approaches or phase field methods, see~\cite{AmbrosioNotes2000} for a theoretical perspective on these approaches. It is worth emphasizing that a wider class of singularities can be handled by the latter two methods, see an example in Figure~\ref{fig:dumbbell}. 
 
 \begin{figure}[htbp]
\centering
\advance \leftskip-2cm\advance \rightskip-2cm
	\includegraphics[width=15cm]{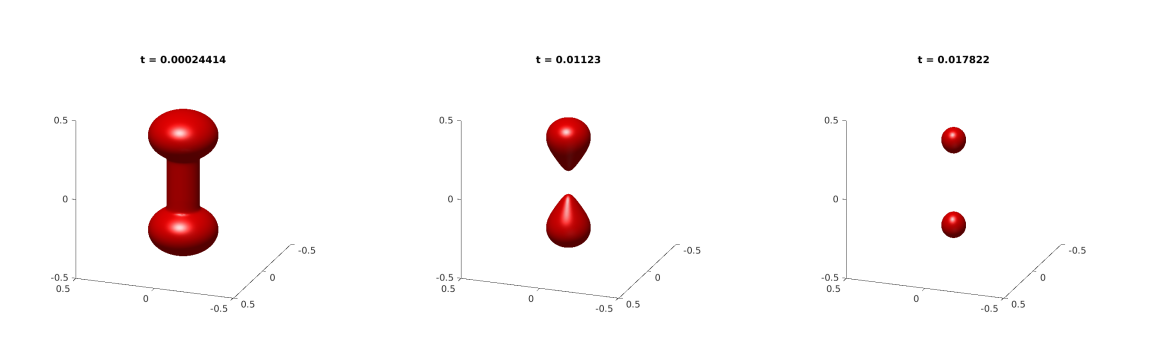} 
	
\caption{
Phase field approximation beyond singularities of the mean curvature flow of a dumbbell.}
\label{fig:dumbbell}
\end{figure}

 In the theory of phase field approximation, the fundamental work of Modica and Mortola~\cite{Modica1977}
 showed that the perimeter of a set can be approximated in the sense of $\Gamma$ convergence (for the $L^1$ topology) by
 the smooth Van der Waals--Cahn--Hilliard energy
 \begin{eqnarray} \label{P_eps}
 P_\varepsilon(u) &=& \ds\int_Q \left( \ds\frac{\varepsilon}{2} |\nabla u|^2 + \ds\frac{1}{\varepsilon} W(u) \right) \,dx.
  \end{eqnarray}
 where $Q \subset \bR^2$ is a sufficiently large fixed bounded domain,
 $\varepsilon>0$ is a small parameter and $W$ is a smooth double-well potential,
 typically
 \begin{eqnarray*}
 W(s) &=& \ds\frac{1}{2}s^2(1-s)^2.
 \end{eqnarray*}
 In particular, the characteristic function of a set $\Omega$ of finite perimeter can be approximated in $L^1$ by functions of the form $u_\varepsilon = q(\frac{\dist(x,\Omega)}{\varepsilon})$ so that $P_\varepsilon(u_\varepsilon) \rightarrow c_W P(\Omega)$ as $\varepsilon\to 0$, with
 $c_W = \int_0^1 \sqrt{2 W(s)} ds$. Here, $q(s) = \frac{1}{2} (1 - \tanh(s))$ is an optimal profile for the problem (a minimizer of the 1D Van der Waals--Cahn--Hilliard energy with suitable values at $\infty$) and $\dist(\cdot,\Omega)$  is the signed distance function to $\Omega$ with the convention that $\dist(\cdot, \Omega)$ takes negative values in the interior of $\Omega$, and nonnegative values outside.
 
%
 The $L^2$-gradient flow of the Van der Waals--Cahn--Hilliard energy $P_\varepsilon$ 
 results in the Allen-Cahn equation \cite{Allen1979}, i.e., up to a time rescaling, 
 \begin{eqnarray}\label{eq_AllenCahn}  
 u_t &=&  \Delta u - \frac{1}{\varepsilon^2}W'(u).
 \end{eqnarray}
 The Cauchy problem for this nonlinear parabolic equation has a unique solution that satisfies a comparison
 principle, see for instance~ \cite[chap. 14]{AmbrosioNotes2000}.
 Furthermore, a smooth set $\Omega$ evolving by mean curvature flow can be approximated by
 $$ 
 \Omega^\varepsilon(t) 
 = \left\{x \in \bR^d,\; u^\varepsilon(x,t) \geq \frac{1}{2} \right\},
 $$
 where $u^\varepsilon$ solves~\eqref{eq_AllenCahn} with initial condition
 $$u^{\varepsilon}(x,0) = q \left( \frac{\dist(x,\Omega(0))}{\varepsilon}\right).$$
 
 A formal asymptotic expansion of $u^{\varepsilon}$ near the interfaces~\cite{BellettiniBook2013} 
 shows that $u^{\varepsilon}$ is quadratically close to the optimal profile, that is, 
 $$ 
 u^{\varepsilon}(x,t) = q \left( \frac{\dist(x,\Omega^{\varepsilon}(t))}{\varepsilon}\right) + O(\varepsilon^2),
 $$
 and the inner normal velocity $V^{\varepsilon}$  of the interface $\{ u^{\varepsilon} = \frac{1}{2}\}$ satisfies
 $$ V^{\varepsilon} = H + O(\varepsilon^2).$$
The convergence of $\partial \Omega_{\varepsilon}(t)$ to $\partial \Omega(t)$ as $\varepsilon \to 0 $ has been rigorously
 proved for smooth flows in \cite{Chen1992, Mottoni1995, Bellettini1995} with a quasi-optimal 
 convergence order $O(\varepsilon^2 | \log \varepsilon |^2)$.
 The fact that $u^{\varepsilon}$ is quadratically close to the optimal profile
 has inspired the development of very effective numerical methods
 in~\cite{BenceMerrimanOsher,Ruuth_efficient,BrasselBretin2011}.

\subsection{Mean curvature flow of thin tubular sets and thickness constraint}
Given a set $\Sigma$ of codimension $2$, the central idea of this paper to approximate the mean curvature flow of $\Sigma$ is to consider a $\sigma$-neighborhood $\Sigma_{\sigma}$ and a phase field mean curvature flow of $\Sigma_{\sigma}$ with a control of the thickness to prevent the disappearance of the surface or any topological change. A theoretical analysis of such an approach is proposed in \cite{BelNov99} under the assumption that the thickness is bounded from below. This assumption is however difficult to adapt numerically, and the novelty of our approach is to introduce a penalization that both constrains implicitly the thickness and is much more suitable for numerical approximation. More precisely, we consider the perturbed Allen-Cahn equation
\begin{equation}
\partial_t u^\varepsilon = 
\Delta u^\varepsilon - \frac{W'(u^\varepsilon)}{\varepsilon^2} (1+ f^{\sigma}),
\label{pac}
\end{equation}
where $f^{\sigma}$ represents the forcing term and acts as a Lagrange multiplier associated with an implicit constraint that the thickness is bounded from below by $\sigma$.

As we will see, the sharp interface limit of this equation gives the following 
law at the boundary of $\Sigma_{\sigma}$: 
$$ V_n = H - \frac{\nabla  f^{\sigma} \cdot n }{2 (1 + f^{\sigma})},$$
where $n$ is the inner unit normal on $\partial\Sigma_{\sigma}$, $V_n$ is the scalar normal velocity and $H$ is the mean curvature.

The main challenge here is the choice of an appropriate form for the forcing term $f^{\sigma}$. To this end, we exploit the properties of the signed distance function and the asymptotic expansion of the solutions to the Allen-Cahn equation, enabling us to locate the skeleton of $\Sigma_{\sigma}$ that we take as an approximation of $\Sigma$. We propose the following form for $f^{\sigma}$:
$$ f^{\sigma}  \simeq  h^\sigma * |  \langle n^\sigma_{u^{\varepsilon}}, (\nabla n^\sigma_{u^{\varepsilon}})^{T} n^\sigma_{u^{\varepsilon}} \rangle|$$
where $n^\sigma_{u^{\varepsilon}} = h^\sigma * \frac{\nabla u^{\varepsilon}}{|\nabla u^{\varepsilon}|}$ and $h^\sigma$ is a kernel, typically a Gaussian kernel of standard deviation $\sigma$. 
We provide theoretical and numerical arguments to show that this particular choice of forcing term allows to locate the skeleton of $\Sigma_{\sigma}$ and, in practice, guarantees a minimal distance of $\sigma$ between $\partial\Sigma_{\sigma}$ and the skeleton. We also provide theoretical and numerical arguments showing the connection of our modified Allen-Cahn equation with the mean curvature flow of $\Sigma$, thus the connection with possibly higher codimension mean curvature flow. In particular, our approach provides an accurate numerical approximation of the mean curvature flow of tubular sets, possibly with endpoints constraints as in the Steiner problem. We will also illustrate that our model can be extended to approximate solutions to the Plateau problem, even in the case where the solution is not orientable. This proves the versatility of our method and its ability to handle complex geometric problems.

\subsection{Outline of the paper}

The paper is organized as follows:

In \textbf{Section 2} we consider a forced mean curvature flow of a domain $\Omega$ with normal velocity $V_n = H + g$. We show that, being appropriately chosen, a forcing term $g$ can prevent the boundary $\partial \Omega$ from approaching a fixed set $\Sigma$ of codimension $2$. 

\textbf{Section 3} is devoted to characterizing the skeleton of $\Omega$ by exploiting the properties of its signed distance function. The key idea is to approximate an expression of the form 
$$Sn = \langle n, \nabla n ^{T} n \rangle,$$ where $n$ is the gradient of the distance function.

In \textbf{Section 4}, we discuss a quasi-static numerical discretization of the perturbed Allen-Cahn equation~\eqref{pac}. The section ends with numerical experiments that illustrate the ability of our approach to model accurately the evolution by mean curvature of filaments.

 Lastly, in \textbf{Section 5} and \textbf{Section 6}, we explain how to couple the perturbed equation with additional inclusion constraints. We deduce numerical approximations of solutions in 3D to the Steiner problem and to the Plateau problem, using for the latter a penalization of the volume.

\section{Approximation of the mean curvature flow with thin fixed obstacles}

Given a closed obstacle $\Sigma\subset\bR^N$ (typically a closed Lipschitz set of codimension $2$ possibly with boundary, but some results proved below remain true for more general sets), the aim of this section is to study how the classical Allen-Cahn equation can be modified to approximate the evolution by mean curvature of the boundary of a smooth bounded open set $\Omega$ constrained to remain at a distance not smaller than $\sigma$  from $\Sigma$.


\subsection{Motivations}
We assume in the following that the obstacle $\Sigma$ is contained in the interior of a smooth open set $\Omega$ with a distance from $\partial \Omega$ at least equal to $\sigma > 0$.  The case where it is contained in the complementary of $\Omega$ can be considered in a similar way. We begin by proving that, the distance from $\Sigma$ to $\partial \Omega$ being positive, the mean curvature at the point of $\partial \Omega$ closest to $\Sigma$ is bounded from above. This result is essential since, if we consider the mean curvature flow with a forcing term $g$, i.e. $V_n = H + g$, it is possible to choose $g$ large enough near $\Sigma$ to be sure that it prevents $\partial \Omega$ from flowing closer to $\Sigma$. Moreover, choosing $g$ close to $0$ far from $\Sigma$ makes it inactive and the constrained flow of $\partial\Omega$ far from $\Sigma$ is very close to the classical mean curvature flow.

\begin{lem}
Let $\Sigma\subset\bR^N$ be a closed set, and $\Omega\subset\bR^N$ be an open, bounded set with $C^2$ boundary such that $\operatorname{dist}(\partial \Omega, \Sigma) =  \delta > 0$. Let $s_0\in \partial \Omega$ be such that $\operatorname{dist}(s_0, \Sigma)=\operatorname{dist}(\partial \Omega, \Sigma)$. Then 
$$H_\Omega(s_0) \leq \frac{N-1}{\delta}.$$
\label{lem:mean_curvature_control}
\end{lem}

\begin{proof}

Thanks to the definitions of $s_0$ and $\delta >0$, there exists a point $p \in \Sigma$ such that $\operatorname{dist}(s_0, \Sigma) = \| s_0 - p \|= \delta > 0$.
Up to some linear transformation on $\bR^N$, we can assume that $p$ has coordinates $(0_{\bR^{N-1}}, \delta)$, $s_0$ has coordinates $(0_{\bR^{N-1}}, 0)$ and the interface $\partial \Omega$ can be locally considered as the graph of a regular fonction $\psi: V\subset\bR^{N-1} \to \bR$ defined on a neighborhood $V$ of $0_{\bR^{N-1}}$ such that $\psi(0_{\bR^{N-1}}) = 0$ and $\nabla \psi(0_{\bR^{N-1}}) = 0_{\bR^{N-1}}$. 
Therefore, the scalar mean curvature of $\partial\Omega$ at $s_0=(0_{\bR^{N-1}}, 0)$
is 
$$
H_\Omega(s_0) =  \operatorname{div} \left( \frac{\nabla \psi}{\sqrt{1+ |\nabla \psi |^2}} \right)(0_{\bR^{N-1}}) = \Delta \psi (0_{\bR^{N-1}})
$$
and we have that
\begin{equation}
\psi(x) = \frac{1}{2} \langle x, \nabla^2 \psi(0_{\bR^{N-1}})x \rangle+ o(\|x\|^2), \quad\mbox{ as } x \rar 0.
\label{eq:taylor_cdp}
\end{equation}
Thanks to the definition of $\delta$, for all $x$ in a neighborhood of $0_{\bR^{N-1}}$, 
\begin{equation}
\|x-0_{\bR^{N-1}}\|^2 + \left( \psi(x) - \delta \right)^2 \geq \delta^2,
\end{equation}
which gives 
\begin{equation}
\|x\|^2 -2 \delta \psi(x) + \psi(x)^2 \geq 0.
\label{eq:taylor_cdp2}
\end{equation}
Hence, by considering $x = \epsilon e_j$ where $\left( e_1, \dots, e_{N-1} \right)$ is the canonical basis of $\bR^{N-1}$and by dividing the above inequality by $\epsilon^2$, as $\epsilon$ goes to $0$, thank to \eqref{eq:taylor_cdp}, we obtain that 
$$
\delta(\nabla^2 \psi(0))_{jj} \leq 1,
$$
which leads to 
$$
H_\Omega(s_0)= \sum_{j=1}^{N-1} \left( \nabla^2 \psi(0) \right)_{jj} \leq \frac{N-1}{\delta}.
$$
\end{proof}

\subsection{The Allen--Cahn equation with a forcing term}
A classical perturbation of the Allen--Cahn equation consists in adding a forcing term $g^{\varepsilon}$ and considering
$$ \partial_t u^\varepsilon =  \Delta u^\varepsilon - \frac{W'(u^\varepsilon)}{\varepsilon^2} + g^\varepsilon.$$
If the forcing term $g^\varepsilon$ satisfies 
\begin{equation}
\sup_{\varepsilon > 0} \int_0^T \int \varepsilon g^\varepsilon(x, t) dx dt < +\infty,
\label{eq:forcing_term_control}
\end{equation}
Mugnai and R\"oger proved in \cite{MugnaiRoger2009} that the above phase field model converges to a forced mean curvature flow with normal velocity
$$
V_n= H +g,
$$
in the sense of varifolds, where $g$ is characterized by 
$$
\lim_{\varepsilon \to 0 } \int_0^T \int_{\bR^N}  \eta \cdot \nabla u^\varepsilon g^\varepsilon \, dx \, dt =
\int_0^T \int _{\bR^N} \eta \cdot g \, d \mu 
,
$$
for every $\eta \in C_c^0((0, T) \times  \bR^N,  \bR^N)$. Here $\mu$ is a Radon measure on $(0, T) \times  \bR^N$ obtained as the limit of \emph{diffuse surface area measures} $\mu^\varepsilon$ defined as 
$$
\mu^\varepsilon := \left( \frac{\varepsilon}{2} |\nabla u^\varepsilon |^2 + \frac{1}{\varepsilon} W(u^\varepsilon) \right) \cL^{N+1}
.
$$

From a variational point of view, recall that the Allen-Cahn equation is the $L^2$-gradient flow of the Cahn-Hilliard energy ${ P}_\varepsilon$ which approximates the perimeter, and the forcing term $g^{\varepsilon}$ can  also be derived from a penalization term  of the form
$$ \varepsilon \int g^{\varepsilon} u dx.$$
However, this penalized model is not well suited to our problem, as it penalizes the whole interior of $\Omega^{\varepsilon}$ (see \cite{Bretin_perrier}) and not only the interface 
$\partial \Omega^{\varepsilon}$. An easy way to focus on the boundary of $\Omega^{\varepsilon}$  involves a weight of the form $W(u^{\varepsilon})$ which is active only near  $\partial\Omega^{\varepsilon}$, and a penalty term of the form
$$ \frac{1}{\varepsilon} \int_{\Sigma}W(u^{\varepsilon}(x,t)) d \cH^{N-2}(x),$$
or a smooth version 
$$ \frac{1}{\varepsilon} \int f^{\sigma}(x) W(u^{\varepsilon}(x,t)) dx,$$
where $f^{\sigma}$ is a nonnegative smooth approximation of $\cH^{N-2}\res\Sigma$, the $(N-2)$-Hausdorff measure supported on $\Sigma$. Such a choice 
leads to the perturbed Allen--Cahn equation 
$$ \partial_t u^\varepsilon = 
\Delta u^\varepsilon - \frac{W'(u^\varepsilon)}{\varepsilon^2} (1+ f^\sigma),$$ 
which has been studied by Qi and Zheng in \cite{QiZheng2018}. In particular, it is proved in  \cite{QiZheng2018} that if $1+ f^\sigma > 0 $  then its sharp interface limit corresponds to a forced mean curvature flow in the sense of varifolds with a vector motion law given by 
\begin{equation}
V= H - \frac{\nabla^\perp  f^\sigma}{2 (1 + f^\sigma)}. 
\label{eq:forced_mean_curvature_flow}
\end{equation}
where $\nabla^\perp  f^\sigma$ is obtained from $\nabla f^\sigma$ by removing its tangential component.

We will now discuss the behavior of the forcing term in \eqref{eq:forced_mean_curvature_flow} and prove that, under certain conditions on $f^\sigma$, \eqref{eq:forced_mean_curvature_flow} provides a flow with obstacle.

\subsection{Choosing the forcing term} 
A natural choice for $f^{\sigma}$ is an approximation of $\cH^k\res{\Sigma}$, for example the convolution of $\cH^k\res{\Sigma}$  with a Gaussian kernel $h^{\sigma}(x)=  \frac{1}{\sigma^N } e^{-\frac{|x|^2}{\sigma^2}}$.

\subsubsection{Case where $\Sigma$ is a single point}
We first consider the simplest case where $\Sigma=\{0\}$ and  the forcing term $f^{\sigma}$ is defined by
$$
f^\sigma(x) = \delta_{0}*h^{\sigma}(x) = \frac{1}{\sigma^N } e^{-\frac{|x|^2}{\sigma^2}}.
$$

\begin{lem} To any $0<\sigma<1$ we associate the function
$
f^\sigma(x) = \sigma^{-N} e^{-\frac{|x|^2}{\sigma^2}}.
$
For all $0<\sigma< \exp(\frac{2(1-N)}N)$, there exists an interval $I_\sigma$ containing $\sigma \sqrt{-N\ln \sigma}$ such that for every open bounded set $\Omega\ni 0$ with smooth boundary and satisfying $\operatorname{dist}(0,\partial \Omega)=\delta>0$ with $\delta\in I_\sigma$, and for $s_0\in \partial \Omega$ such that $|s_0|=\operatorname{dist}(0,\partial \Omega)=\delta$, there holds
$$H_\Omega(s_0) - \frac{\nabla  f^\sigma \cdot n }{2 (1 + f^\sigma)}(s_0)<0$$
where $n(s_0)$ is the inner unit normal to $\partial\Omega$ at $s_0$.
\label{claim:minimal_distance}
\end{lem}

\begin{rmk}
This result implies that, for a time-dependent set $\Omega(t)$ evolving by the forced mean curvature motion \eqref{eq:forced_mean_curvature_flow} and such that $0 \in \Omega^-$, the forcing term $f^\sigma$ acts as a repulsive source in a ring around $0$. These observations can be extended to the case where the obstacle is a general set $C^2$.

\end{rmk}

\begin{rmk}

One of the advantages of using the forced flow \eqref{eq:forced_mean_curvature_flow} is that when the interface is far from the obstacle, it behaves almost exactly like a mean curvature flow. When the interface starts to approach the obstacle, the perturbed term prevents the interface from self-colliding by forcing the interface and the obstacle to remain beyond a minimum distance, denoted by $\tilde{\sigma}$ as a function of $\sigma$ when the interface becomes stationary. It is natural to extend this idea to dynamic obstacles, and even to obstacles provided by the interface itself, typically its skeleton. This is the subject of the next section.
\end{rmk}

\begin{proof}[Proof of Lemma~\ref{claim:minimal_distance}]
By the definitions of $\delta$ and $s_0$, we first notice that 
$$\nabla f^\sigma (s_0) \cdot n(s_0) = \partial_{n(s_0)} f^\sigma$$ where $$n(s_0)= - \frac{s_0}{|s_0 |} = -\frac{s_0}{\delta}.$$
Therefore, 
$$
f(s_0) =\frac{1}{\sigma^N} e^{-\frac{\delta^2}{\sigma^2}} \qquad\mbox{ and } \qquad
\nabla f^\sigma(s_0) \cdot n(s_0) =  \frac{2\delta}{\sigma^{N+2}} e^{-\frac{\delta^2}{\sigma^2}}.
$$
Thanks to Lemma \ref{lem:mean_curvature_control}, it follows that
$$H_\Omega(s_0) - \frac{\nabla  f^\sigma \cdot n }{2 (1 + f^\sigma)}(s_0)\leq \frac{N-1}{\delta} -
\frac{  \frac{\delta}{\sigma^{(N+2)}} e^{ -\frac{\delta^2}{\sigma ^2}}}{ 1 + \frac{1}{\sigma^N} e^{ -\frac{\delta^2}{\sigma^2}}}
.$$
In order to have the left-hand term negative, it is sufficient to ensure that,  
$$
w(\delta) := \frac{\frac{\delta^2}{\sigma^{(N+2)}} e^ {-\frac{\delta^2}{\sigma^2}}}{ 1 + \frac{1}{\sigma^N} e^{-\frac{\delta^2}{\sigma^2}}} > N-1.
$$
Since $w$ is continuous on $[0, +\infty[$, satisfies $\lim_{\delta \rar +\infty} w(\delta) = w(0) = 0$,  and as $0<\sigma <1$,  
$$
\max_{[0, +\infty[} w \geq w(\sigma \sqrt{-\ln \sigma^N}) 
= \frac{-\ln \sigma^N}{2}
.
$$
we deduce from the intermediate value theorem that for all $0<\sigma<\exp{\frac{2(1-N)}N}$ there exists an interval $I_\sigma$ containing $\sigma \sqrt{-\ln \sigma^N}$ such that for every $\delta\in I_\sigma$, $w(\delta)>N-1$, and the lemma ensues.
\end{proof}

\subsubsection{Case of a general fixed obstacle}

Thanks to Lemma~\ref{lem:mean_curvature_control}, we show in the following that, for $\sigma>0$ small enough, we can avoid the interface colliding with a given \emph{static obstacle}.

\begin{thm}
Let $\Sigma \subset \bR^n$ be a closed $k$-rectifiable set and, for $0<\sigma<1$, let 
$
f^\sigma = h^\sigma * \cH^k\res \Sigma
$. Let $t\in[0,T]\mapsto \Omega(t)$ be the smooth mean curvature flow of a smooth open, bounded set. There exists $0<\sigma<1$ such that, if $\Sigma\subset\Omega(0)^-$ (the interior of $\Omega(0)$) and $\operatorname{dist}(\partial \Omega(0), \Sigma) \geq \delta>0$,  then for every $t \in [0,T]$,
$
\operatorname{dist}(\partial \Omega(t), \Sigma) \geq  \delta.
$
\label{thm:no_contact}
\end{thm}

\begin{proof}
Let $t >0$ be such that $\operatorname{dist}(\partial \Omega(t), \Sigma) = \delta$. Thanks to Lemma \ref{lem:mean_curvature_control} and the definition of the forced mean curvature flow \eqref{eq:forced_mean_curvature_flow}, for every $s_t \in \partial \Omega(t)$ satisfying $\operatorname{dist}(\partial \Omega(t), \Sigma) = \operatorname{dist}(s_t, \Sigma) = \delta$, we have that 
$$
V_n(s_t) = H_{\Omega(t)}(s_t) - \frac{\nabla f^\sigma(s_t) \cdot n(s_t)}{2(1+ f^\sigma(s_t))} \leq \frac{N-1}{\delta} - \frac{\nabla f^\sigma(s_t) \cdot n(s_y)}{2(1+ f^\sigma(s_t))},
$$
Assume that $\Sigma \subset \Omega(t)^{-}$. Then
%
\[
\lim_{\sigma\to 0} \frac{\nabla f^\sigma (s_t) \cdot n(s_t)}{f^\sigma(s_t)} =+\infty,
\]
therefore, for $\sigma_t>0$ small enough, we get that 
\[
V_n(s_t) \leq  \frac{N-1}{\delta} - \frac{\nabla f^{\sigma_t}(s_t) \cdot n(s_t)}{2(1+ f^{\sigma_t}(s_t))} < 0.
\]
By contradiction, assume now that $\Sigma \cap (\Omega(t)^{-})^c$ is not empty. By the regularity theory, the flow is continuous in space so there exists an earlier time $0\leq t'<t$ such that 
$\Sigma \subset \Omega(t')^{-}$ and $\operatorname{dist}(\partial \Omega(t'), \Sigma) = \dist(s',\Sigma)=\delta$ with $s'\in \partial \Omega(t')$. By the above argument, $V_n(s')<0$ for a suitable $\sigma_{t'}$ so the distance with $\Sigma$ cannot decrease, which contradicts the original assumption. It follows that, for every $t\in[0,T]$ there exists $\sigma_t>0$ and a closest point $s_t$ of $\partial\Omega(t)$ to $\Sigma$ such that $V_n(s_t)<0$. Using the smoothness of $\frac{\nabla f^{\sigma_t}(s) \cdot n(s)}{2(1+ f^{\sigma_t}(s))}$ on a suitable open set containing $\Sigma$ and strictly contained in $\Omega(t)^-$ for all $t$, we deduce the existence of $\sigma>0$ such that $V_n(s_t)<0$ for every $t \in [0,T]$, hence, the distance between $\partial \Omega$ and $\Sigma$ is bounded below by $\delta$, which completes the proof.
\end{proof}


\section{The skeleton and its approximation}

There is a rich literature on {\it medial axes}, {\it skeletons} and {\it cut locus} of sets with various definitions. For it is more suitable for what we need later, we define the skeleton as follows (as in~\cite{AmbrosioNotes2000} where it is called the medial axis):
\begin{Def}
The \emph{skeleton} of $\Omega \subset\bR^N$ is the singular set $\operatorname{S}_\Omega$ of $d_E$ defined as
\[\operatorname{S}_\Omega=\{x\in \bR^N, \; \dist(x,\Omega) \text{ is not differentiable at }x\}\]
\label{def:skeleton}
\end{Def}
\begin{rmk}
It follows that $\nabla \dist(x,\Omega)$ is well defined on $\bR^N\setminus \operatorname{S}_\Omega$ and takes values in $\bS^{N-1}$.
In addition, $x \in \operatorname{S}_\Omega\setminus\partial \Omega$ if and only if there exist at least two distinct points $y_1, y_2 \in \partial \Omega$ such that $|x- y_1| = |x- y_2|$.
\end{rmk}

The following regularity properties of the skeleton are proved in~\cite{MennucciMantegazza} for the distance function to a closed, nonempty set, and easily extend to the signed distance function $\dist(x,\Omega)$ associated with a general set $\Omega \subset\bR^N$ with nonempty boundary.
\begin{prop}
If $\Omega \subset \bR^N$ has nonempty boundary then
\begin{itemize}
\item $\operatorname{S}_\Omega \setminus \partial \Omega$ is $C^{2}-(N-1)$--rectifiable.
\item If, in addition, $\partial \Omega$ is of class $C^r$ with $r\geq 3$ then the closure $\overline{\operatorname{S}_\Omega}$ of $\operatorname{S}_\Omega$ is $C^{r-2}-(N-1)$--rectifiable. 
In particular,
\begin{itemize}
\item the Hausdorff dimension of $\overline{\operatorname{S}_\Omega}$ is at most $(N-1)$.
\item the vector field $\nabla \dist(x,\Omega)$ belongs to the space $\SBV_{\!\text{\small loc}}(\bR^N)$.
\end{itemize}
\end{itemize}
\end{prop}
\begin{rmk} Given a general set  $\Omega \subset \bR^N$ with nonempty boundary, $\overline{\operatorname{S}_\Omega}$ is not $(N-1)$-rectifiable in general. There is an example in~\cite{MennucciMantegazza} of a convex open set in $\bR^2$ with $C^{1,1}$ boundary such that the closure of its skeleton has positive Lebesgue measure.
\end{rmk}
The boundary of $\Omega$ is the set of points where $\dist(x,\Omega)$ vanishes. Is there a similar way to link the skeleton of $\Omega$ with the (complementary of the) support of a map defined either with $\dist(x,\Omega)$ or $\nabla \dist(x,\Omega)$? There is no easy way to do it directly, but the skeleton can be obtained as (a subset of) the support of the distributional limit of a sequence of well defined approximate maps. 

There are however essentially two difficulties with the skeleton:
\begin{itemize}
\item A smooth shape may have a nonsmooth, and possibly very nonsmooth skeleton~\cite{MennucciMantegazza}. 
\item There is no continuous dependence of the skeleton of a set $\Omega$ with respect to perturbations of $\Omega$. A disk in $\bR^2$ has its center as skeleton but it is easy to design arbitrarily small perturbations of the disk's boundary to get a collection of skeletons with Hausdorff dimension $1$ and whose Hausdorff distance to the disk's center is bounded from below by a positive number.
\end{itemize}
We shall see however that using a skeleton-based term yields an interesting and efficient model for the numerical approximation of the mean curvature flow of thin structures.

The purpose of the next section is to provide a constructive formula for approximating a weighted Hausdorff measure supported on the skeleton $\Sigma = \operatorname{S}_{\Omega}$ of a set $\Omega$. 

\subsection{Skeleton approximation with a constructive formula} 
The vector field $ n (x) = \nabla \dist(x,\Omega)$ satisfies the equality 
$|n|^2 =1$ at every point where is it defined. In particular, if $n$ is regular at $x$, then 
$$
S n(x) :=  \langle \nabla n^T(x) n(x) , n(x) \rangle = 0,
$$
where $\nabla n^T$ is the transpose of the matrix $\nabla n$ and $\langle \cdot,\cdot \rangle$ is the inner product. Of course, this term cannot be defined if $x$ is on the skeleton $\Sigma = S_{\Omega}$, but a suitable smoothing will charge the support of $\Sigma$. The term defines above gives rise to the following definition.
\begin{Def}
For every $ n \in SBV(\bR^N,\bR^N)$, we define the  \emph{skeletal term} $S n$ of $ n$ as
$$
S  n
= \langle n ,(\nabla n )^T n \rangle.$$
\label{def:jump_term}
\end{Def}
The idea is then to look at a regularized version of this term by applying this expression to the vector field $n^{\sigma} = n *h^{\sigma}$ corresponding to the convolution of $n$ with the kernel $h^{\sigma} = \frac{1}{\sigma^N} h\left(\frac{\cdot}{\sigma} \right)$ 
where $h \in C^\infty_c (\bR^N,\bR^+)$  and $\int_{\bR^N} h dx = 1$.

\begin{thm}
Let $ n \in SBV(\bR^N, \bS^{N-1})$ with $C^1-$ jump set $\Sigma$ oriented $\cH^{N-1}$-a.e. by the unit vector $\nu$.
Then \begin{equation}
S  n^\sigma  \to 
\frac{1}{12} |[ n]|^2 \langle [ n], \nu \rangle \, \cH^{N-1}\res\Sigma
\quad\text{ in } \cD'(\bR^N) 
\quad\text{ as } \sigma \to 0
,
\label{eq:jump_term_convergence}
\end{equation}
 where $[n]: = n^+ - n^-$ with  $n^+, n^-$ the approximate limits of $n$ with respect to $\nu$ on $\Sigma$, i.e.
  $$ 
  n^\pm (x) := \lim_{r\to 0} \fint_{B^\pm(x, r)} u(y) \, dy
  $$
where $B^+(x, r) = \{ y \in \bR^N \, | \, \langle y, \nu(x) \rangle > 0 \}$ and 
$B^-(x, r) = \{ y \in \bR^N \, | \, \langle y, \nu(x) \rangle < 0 \}$. 

\label{thm:jump_term}
\end{thm}

\begin{rmk}
 For every $u \in SBV(\bR^N)$, $Du$ is a Radon measure that can be written as the sum of an absolutely continuous part and a singular part with respect to the Lebesgue measure, i.e,
$$
Du = D^A u + D^J u.
,
$$
where $D^A u = \nabla^A u\, dx$ with $\nabla^A u$ the Radon-Nikodym derivative of $Du$ with respect to the Lebesgue measure $\cL^N$ in $\bR^N$, and $D^Ju=[u]\nu\cH^{N-1}\res\Sigma$ with $\Sigma$ the jump set of $u$, $\nu$ a unit normal and $[u]$ the jump of $u$ on $\Sigma$, see~\cite{AmbrosioFuscoPallara2000}. Recall also that $u$ is approximately differentiable a.e. and its approximate differential coincides with $\nabla^A u$ a.e. in $\Omega$.
\end{rmk}

\begin{proof}
Since $  n^\sigma = h^\sigma * n$, 
we have that 
$$
\begin{aligned}
\partial_j   n^\sigma_i = h^\sigma * \partial_j  n_i
= h^\sigma * \partial_j^A  n_i +  h^\sigma * \partial_j^J  n_i
=
h^\sigma * \partial_j^A  n_i +
h^\sigma * [  n_i ] \nu_j \, \cH^{N-1}\res\Sigma.
\end{aligned}
$$
Therefore, for every function with compact support $\varphi \in C^\infty_c (\bR^N)$, we can write that
$$
\sum_{i, j} \int_{\bR^N} n^\sigma_i \partial_j  n^\sigma_i  n^\sigma_j \varphi dx 
= I_1^\sigma + I_2^\sigma
,
$$
where 
\begin{equation}
  I_1^\sigma = \sum_{i, j}\int_{\bR^N}   n_i^{\sigma} (h^\sigma * \partial_j^A   n_i ) n^\sigma_j \varphi \, dx\;\;
\text{ and }\;\;
I_2^\sigma = \sum_{i, j}\int_{\bR^N}   n^\sigma_i g^\sigma_{ij}   n^\sigma_j \varphi \, dx
\label{pf1:jump_term}
\end{equation}
with 
$$
g^\sigma_{ij} = h^\sigma * [ n_i] \nu_j \, \cH^{N-1}\res\Sigma.
$$
\par
Without loss of generality, we can consider the restriction of these integrals to the set  $Q =\Sigma(1) := \{ x \in \Omega \, | \, d(x, \Sigma) \leq 1 \}$ and assume that the map 
$$
\begin{aligned}
\tau:\Sigma(1) &\to [-1, 1] \times \Sigma
\\
x &\mapsto (r, s)
\end{aligned}
$$
is a diffeomorphism, with $s = \pi_\Sigma(x)$ the normal projection of $x$ on $\Sigma$ and $r = \langle x - s, \nu(s) \rangle$, so that $x = s+ r\nu(s)$, . 
Thanks to the fact that
$$  n^\sigma_i = h^\sigma *   n_i \to  n_i \;\;\text{ a.e.,}\qquad
h^\sigma * \partial_j^A  n_i \to \partial^A_j  n_i  \;\text{ a.e. as } \sigma \to 0,
$$ and $|  n |^2 =\sum_i  n_i^2= 1$, 
we get that, for every $\sigma>0$, $|  n_i^\sigma (h^\sigma * \partial^A_j  n_i ) | \leq C |\partial^A_j  n_i |$ for some $C>0$ and
\begin{equation}
\sum_i  n^\sigma_i  
\left (h^\sigma * \partial_j^A  n_i \right ) \to \sum_i  n_i \partial^A_j  n_i = \frac{1}{2} \partial_j \left (|  n |^2 \right ) = 0  \;\; \text{ a.e. as } \sigma \to 0,
\end{equation} 
which, thanks to Lebesgue's theorem, leads to
\begin{equation}
\lim_{\sigma \to 0}
I^\sigma_1 = \sum_{i, j} \int_Q  n_i \partial_j^A  n_i  n_j \varphi \, dx =0.
\label{pf2:jump_term}
\end{equation}
In order to estimate $I^\sigma_2$ defined in \eqref{pf1:jump_term}, by identifying $x = (r, s)_\tau$ and taking $z=r/\sigma$, we can write that
\begin{equation}
\begin{aligned}
I^\sigma_2 &= 
\sum_{i, j} \int_{-1}^1 \int_\Sigma  n_i^\sigma (r, s) g^\sigma_{ij} (r, s)  n_j^\sigma (r, s) \varphi(r, s) J_\Sigma(r, s) \, ds \, dr
\\
&= \sum_{i, j} \int_{-\infty}^{\infty} \int_\Sigma  n_i^\sigma (\sigma z, s)  g^\sigma_{ij} (\sigma z, s)  n_j^\sigma (\sigma z, s) \varphi(\sigma z, s) \chi_{[-1/\sigma, 1/\sigma]}(z)\sigma J_\Sigma (\sigma z, s) \, ds \, dz
,
\end{aligned}
\label{eq:I_2}
\end{equation}
where $\chi_{[-1/\sigma, 1/\sigma]}$ is the characteristic function of $[-1/\sigma, 1/\sigma]$ and $J_\Sigma = |\mdet \left( \partial_{(r, s)} \tau^{-1} (r, s)_\tau\right)|$ is the Jacobian of $\tau^{-1}$.
\par 
We are now at the point of studying the pointwise limit of the integrand function as $\sigma$ goes to $0$.
Let $(\sigma z_0, s_0)_\tau \in [-1, 1] \times \Sigma$, we have that
\begin{equation}
\sigma g_{ij}^\sigma (\sigma z_0, s_0)_\tau
= \int_{\Sigma} \sigma h^\sigma((\sigma z_0, s_0)_\tau - (0, s)_\tau ) [ n_i](s) \nu_j (s) ds.
\label{pf3:jump_term}
\end{equation}
As $\Sigma$ is $C^1$, we can parametrize locally $\Sigma$ as a graph around $s_0$. More precisely, up to rotation, we can assume that $\nu(s_0)= (0_{\bR^{N-1}}, 1)$ and then there exist a neighbourhood $V(s_0)$ of $s_0$ and $\delta > 0$ such that for all $s \in V(s_0)$, we can write that
\begin{equation}
s= s(u) = s_0 + (u, \Psi(u)) 
,
\label{pf4:jump_term}
\end{equation}
where $u\in B^{N-1}(0, \delta)$ and $\Psi: B^{N-1}(0, \delta) \to \bR$ is a $C^1$-map such that $\Psi(0)= | \nabla \Psi(0) |= 0$. For $\sigma>0$ small enough, \eqref{pf3:jump_term} can thus be written as 
\begin{equation}
\begin{aligned}
\sigma g^\sigma_{ij}(\sigma z_0, s_0) &= 
\int_{B^{N-1}(0, \delta)} \sigma h^\sigma\left( (\sigma z_0, s_0)_\tau - (0, s(u))_\tau \right) [n_i]\nu_j \left( s(u) \right) \sqrt{ 1 + |\nabla \Psi(u)|^2}\, du+o(1)
\\
&= \int_{B^{N-1}(0, \delta/\sigma)} h\left(\frac{(\sigma z_0, s_0)_\tau - (0, s(\sigma u))_\tau }{\sigma} \right) [ n_i] \nu_j (s(\sigma u)) \sqrt{1 + |\nabla \Psi(\sigma u)|^2}\,  du+o(1)
,
\end{aligned}
\label{pf5:jump_term}
\end{equation}
where the second equality is obtained by replacing $u$ with $\sigma u$.
Moreover we have that 
\begin{equation}
\begin{aligned}\frac{(\sigma z_0, s_0)_\tau - (0, s(\sigma u))_\tau }{\sigma} &= 
\frac{s_0 + (0_{\bR^{N-1}}, \sigma z_0 ) - s_0 - (\sigma u, \Psi(\sigma u)) }{\sigma} 
\\
&= \left (-u, z_0 - \frac{\Psi(\sigma u)}{\sigma} \right )
\to (-u, z_0) \text{ as } \sigma \to 0.
\end{aligned}
\label{pf6:jump_term}
\end{equation}
Together with \eqref{pf5:jump_term} and \eqref{pf6:jump_term},  we get that
\begin{equation}
\sigma g_{ij}^\sigma (\sigma z_0, s_0) \to g(z_0) [ n_i](s_0) \nu_j (s_0)
 \text{ as } \sigma \to 0,
\label{eq:g_sigma_limit}
\end{equation}
where $g(z_0) = \int_{\bR^{N-1}} h(-u, z_0) du = \int_{\bR^{N-1}} h(u, z_0)\,  du$.
\par
By analogy, for $\sigma$ small enough, we also have
\begin{equation}
\begin{aligned}
  n_i^\sigma (\sigma z_0, s_0) &=
\int_{-1/\sigma}^{1/\sigma} \int_{\Sigma}
h^\sigma \left( (\sigma z_0, s_0)_\tau - (\sigma z, s)_\tau \right)  n_i (\sigma z, s)_\tau \sigma J_\Sigma (\sigma z, s)_\tau \, ds \, dz
\\
&=
\int_{-1/\sigma}^{1/\sigma} \int_{B^{N-1}(0, \delta/\sigma)}
h\left( \frac{(\sigma z_0, s_0)_\tau - (\sigma z, s(\sigma u))_\tau}{\sigma} \right)
\\
&\times  n_i (\sigma z, s(\sigma u))_\tau J_\Sigma (\sigma z, s(\sigma u)_\tau ) \sqrt{1 + | \nabla \Psi(\sigma u) |^2} \, du \, dz
\end{aligned}
\label{pf7:jump_term}
\end{equation}
Moreover, as $\sigma \to 0$,
\begin{equation}
\begin{aligned}
  n_i (\sigma z, s(\sigma u))_\tau &\to 
\left \{
\begin{aligned}
 n_i^+ (s_0) &\text{ if } z >0,
\\
 n_i^- (s_0) &\text{ if } z < 0,
\end{aligned}
\right .
\\
J_\Sigma (\sigma z, s(\sigma u))_\tau &\to 1
\text{ and }
\\
\frac{(\sigma z_0, s_0)_\tau - (\sigma z, s(\sigma u))_\tau}{\sigma} &\to \left( -u, z_0 - z  \right).
\end{aligned}
\label{pf8:jump_term}
\end{equation}
Notice that the second limit is due to the fact that $J_\Sigma(r, s)_\tau =  1 + H_\Sigma (s)r + o(r)$, see~\cite{LeonSimon}.
\par 
Thanks to \eqref{pf7:jump_term} and \eqref{pf8:jump_term}, we get that
\begin{equation}
\begin{aligned}
 n_i (\sigma z_0, s_0) &\to 
\int_0^{+\infty} \int_{\bR^{N-1}} h(-u, z_0-z)  n_i^+(s_0)\,  du \, dz +
\int_{-\infty}^0 \int_{\bR^{N-1}} h(-u, z_0 - z)  n_i^- (s_0) \, du \, dz
\\
&=
\left (\int_{-\infty}^{z_0} g(z) \, dz \right )  n_i^+ (s_0)
+
\left ( \int_{z_0}^{+\infty} g(z) \, dz \right )  n_i^- (s_0)
\\
&=G(z_0)  n_i^+(s_0) + \left( 1 - G(z_0) \right)  n_i^- (s_0)
,
\end{aligned}
\label{eq:n_sigma_limit}
\end{equation}
as $\sigma \to 0$,
where $G(z_0) = \int_{-\infty}^{z_0} g(z) dz$.
\par
Lastly, together with \eqref{eq:I_2}, \eqref{eq:g_sigma_limit}, \eqref{pf8:jump_term} and \eqref{eq:n_sigma_limit}, we obtain that
\begin{equation}
\begin{aligned}
\lim_{\sigma \to 0} I_2^\sigma =&
\sum_{i, j} 
\int_\Sigma \int_\bR
\left ( G(z)  n_i^+(s) + \left( 1 - G(z) \right)  n_i^- (s) \right )  (g(z) [ n_i] \nu_j)
\\
&\quad \quad \times
\left ( G(z)  n_j^+(s) + \left( 1 - G(z) \right)  n_j^- (s) \right )\, dz \, \varphi(0, s)_\tau \, ds
\\
=& \sum_{i} 
\int_\Sigma \int_\bR  
G(z) g(z) [ n_i] (G(z)  n^+_i + (1-G(z)) n^-_i)
\langle  n^+, \nu \rangle (s) \varphi(0, s)_\tau \, ds
\\
&+
\sum_i \int_\Sigma \int_\bR
(1 -G(z) )g(z) [ n_i] (G(z)  n^+_i + (1-G(z)) n^-_i)
\langle  n^-, \nu \rangle (s) \varphi(0, s)_\tau \, ds
,
\end{aligned}
\end{equation}
which, 
thanks to the fact that $\sum_i \left( 1 -  n_i^+  n_i^- \right) = \frac{| [ n] |^2}{2}$, leads to
\begin{equation}
\begin{aligned}
\lim_{\sigma \to 0} I_2^\sigma 
=&
\sum_i \int_\Sigma \int_\bR G(z) g(z) (2G(z) -1) (1 -  n_i^+  n_i^i)
\langle  n^+, \nu \rangle (s) \varphi(0, s)_\tau \, ds
\\
&+
\sum_i \int_\Sigma \int_\bR (1-G(z)) g(z) (2G(z) -1) (1 -  n_i^+  n_i^i)
\langle  n^-, \nu \rangle (s) \varphi(0, s)_\tau \, ds
\\
=&
\int_\Sigma \left( \int_\bR G(z) g(z) (2G(z) -1 ) \, dz \right)
\frac{|[ n]|^2}{2}
\langle  n^+, \nu \rangle (s) \varphi(0, s)_\tau \, ds
\\
&+
\int_\Sigma \left( \int_\bR (1-G(z))g(z)(2G(z) -1 )\, dz \right)
\frac{|[ n]|^2}{2}
\langle  n^-, \nu \rangle (s) \varphi(0, s)_\tau \, ds
\\
=&\frac{1}{12} \int_\Sigma |[ n]|^2 \langle [ n], \nu \rangle \varphi(0, s) \, ds
= \left \langle \frac{1}{12} |[ n]|^2 \langle [ n], \nu \rangle \, \cH^{N-1}\res\Sigma, \varphi \right \rangle
.
\end{aligned}
\label{eq:I_2_limit}
\end{equation}
and Theorem \ref{thm:jump_term} ensues.
\end{proof}

\begin{rmk}
In the case where $N=2$ and $ n =  \nabla \dist(x,\Omega)$ is the gradient of the signed distance function of a set $\Omega$,  
 $ n^\perp := 
\begin{pmatrix}
-\partial_{x_2} \dist(x,\Omega) 
\\
\partial_{x_1} \dist(x,\Omega)
\end{pmatrix}$ satisfies 
$$
\nabla \cdot  n^\perp = 0
\text{ in } D'(\Omega),
$$ (or equivalently, $\nabla \times  n =0 \text{ in } D'(\Omega)$). Thanks to the divergence-free hypothesis and the trace theory of functions of bounded variation,
we can write that 
$$
 n^+ = \cos \theta \nu^\perp + \sin \theta \nu
\text{ and }
 n^- = \cos \theta \nu^\perp - \sin \theta \nu
$$
on the jump set $\Sigma( n)$,
where $\theta$ is the angle between $ n^+$ and $\nu^\perp$.
Therefore, we have that $[  n] =  2 \sin \theta  = \langle [ n], \nu \rangle$, which, combined with Theorem \ref{thm:jump_term}, implies that
$$
\lim_{\sigma \to 0} S  n^\sigma = \frac{1}{12} \langle [ n], \nu \rangle^3 \cH^{1}\res\Sigma \text{ in } D'(\Omega).
$$
\end{rmk}
%

\subsection{Numerical illustrations }
We now provide numerical illustrations of the previous theorem for 2D and 3D examples. To do so, we compute numerical approximations of $S  n^\sigma$ on a Cartesian grid  with $ n^\sigma =  h^\sigma *  n$ and $h^{\sigma} = \frac{1}{\sigma^N} e^{-\pi \frac{|x|^2}{\sigma^2}}$ a Gaussian kernel. 

\subsubsection{2D examples}
We provide in Figure \ref{Jump2D} illustrations of the skeletal term for different values of the regularization parameter $\sigma> 0$ and for three 2D sets whose boundaries consist of, respectively,
\begin{itemize}
\item two concentric circles,
\item two parallel lines, 
\item a rectangle.
\end{itemize}
The resolution is $2^7$ nodes for each grid dimension. Clearly, $S n^\sigma$ identifies the (codimension 1) main part of the skeleton $\Sigma$ of $\Omega$, with a weight that depends on the orientation $\nu$ of $\Sigma$. Negligible parts of the skeleton (i.e. of codimension greater than 1) are also detected, for example the centre of the ring is activated.

\par
 \begin{figure}[!htpb]
\centering
 	\includegraphics[width=14cm]{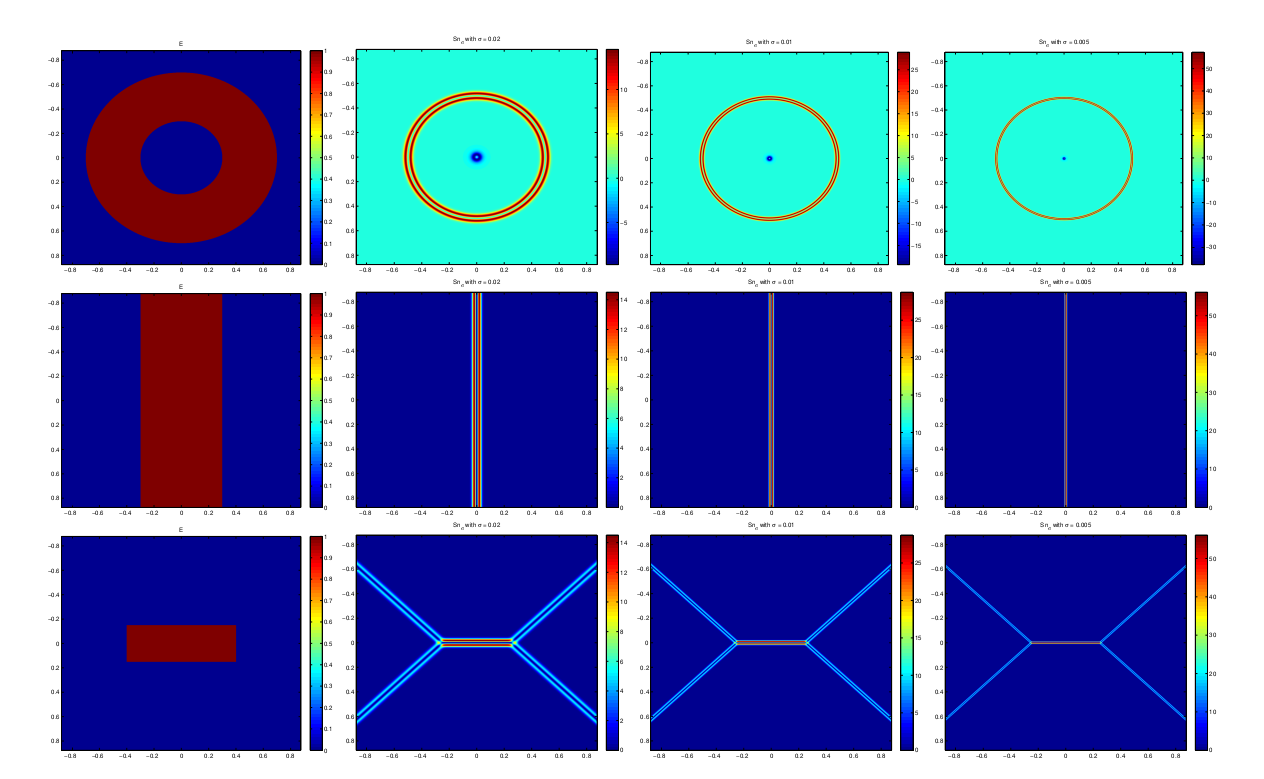} 

\caption{Numerical values of $S  n^\sigma$ in $2D$ for various configurations and various values of $\sigma$;
The first column of each line corresponds to the considered set ; On each line, columns $2$, $3$ and $4$ show the values of the scalar field 
$S  n^\sigma$ for, respectively, $\sigma = 0.02$, $0.01$ and $0.005$.}
\label{Jump2D}
\end{figure}

 \begin{figure}[!htpb]
\centering
		\includegraphics[width=18cm]{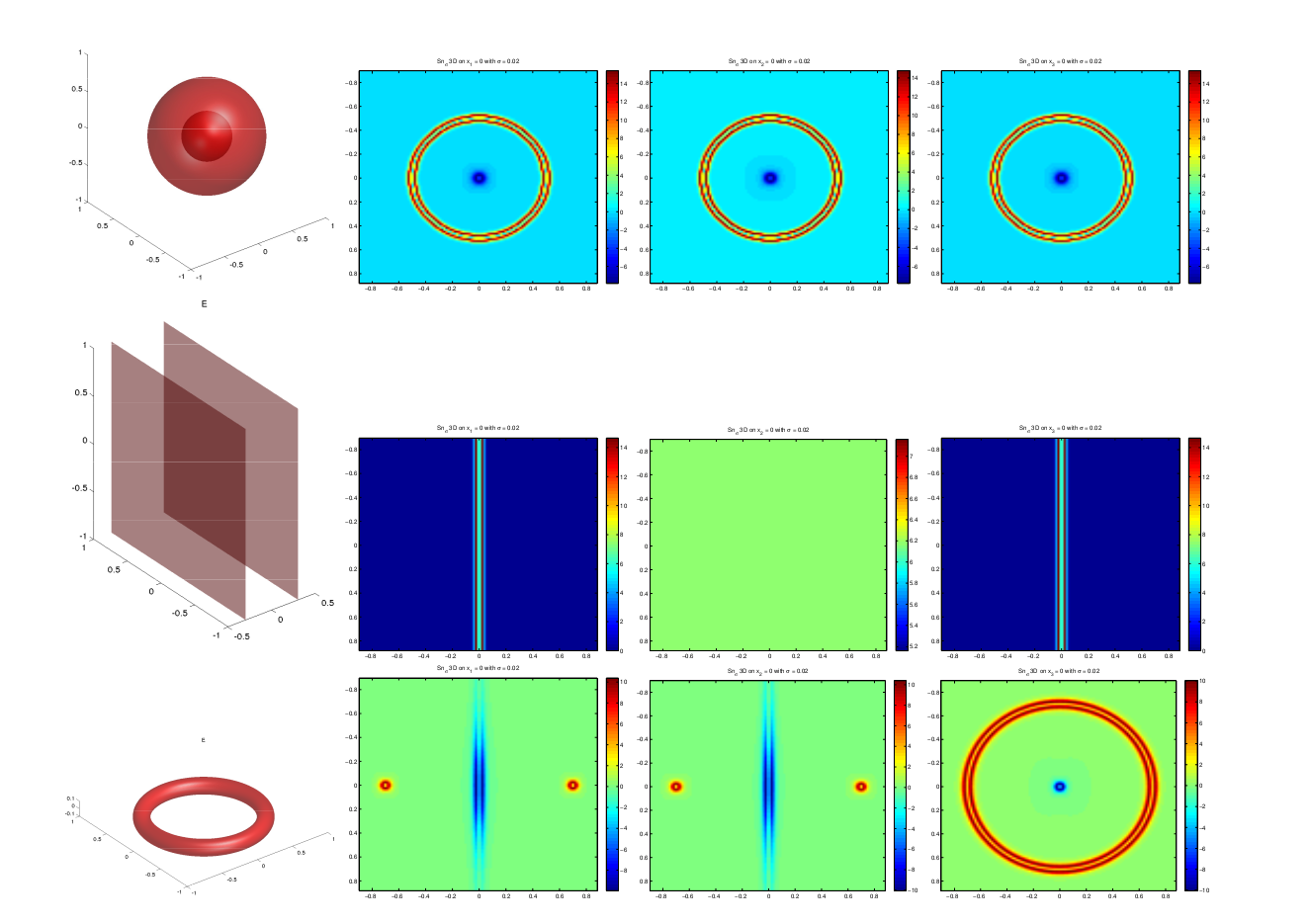} 
\caption{Numerical values of $S  n^\sigma$ for $3D$ examples. Each line shows the considered set, then the values taken by the skeletal term on the planes $\{x_1 = 0\}$, 
$\{x_2 = 0\}$, and $\{x_3 = 0\}$, respectively.}
\label{Jump3D}
\end{figure} 

\subsubsection{3D examples}
We perform a similar test in 3D with a numerical resolution of $2^7$ nodes
for each grid dimension. The examples used for the illustration in Figure~\ref{Jump3D} have as boundaries
\begin{itemize}
\item two concentric spheres, 
\item two parallel planes, 
\item a circular torus.
\end{itemize}
The values taken by the skeletal term are captured on the three hyperplanes $\{x_1 = 0\}$, $\{x_2 = 0\}$ and $\{x_3 = 0\}$. 
 
 \par
 In both figures \ref{Jump2D} and \ref{Jump3D}, one can clearly observe that the skeletons of all considered sets are well localized.
 
\subsection{Asymptotic expansion of $S n^{\sigma}$ with respect to $\sigma$ }
Thanks to Theorem \ref{thm:jump_term}, the skeletal term $S  n^\sigma$ approximates the codimension $1$ jump set of $ n$ weighted by the height of the jump and the angle with the normal to the set. However, $S  n^\sigma$ does not necessarily charge at the limit as $\sigma$ goes to $0$ singular parts of higher codimension. For instance, consider $ n(x) =  - x/|x| \in W^{1, 1}(\bR^2, \bS^1)$.  As shown in the previous section, the skeletal term satisfies
$$
S  n^\sigma \to 0 \text{ in } \cD'(\Omega) \text{ as } \sigma \to 0.
$$
Indeed, since the singular part of $D  n$ vanishes, thanks to \eqref{pf1:jump_term} and \eqref{pf2:jump_term}, we have that for every $\varphi \in C^\infty_c(\Omega)$
$$
\langle S  n^\sigma, \varphi \rangle = I_1^\sigma 
\to 0 \text{ as } \sigma \to 0.
$$

We will focus in this section on two specific examples, a round sphere in $\bR^N$ and a circular torus in $\bR^3$, to show that, at least in these two cases,  higher codimension parts of the skeleton are also activated by $S  n^\sigma$.  From this preliminary analysis
and the previous theorem, we  conjecture that if $ n$ is the gradient of the signed distance function associated with a set $\Omega$,
its skeletal term $S  n^\sigma$ characterizes asymptotically the jump set $\Sigma = S_{\Omega}$, which corresponds to the skeleton of $\Omega$. More precisely, it is tempting to propose the following conjecture:

\begin{conj}
\begin{equation}
{Sn^\sigma\sim \sum_{j=0}^{N-1}{\sigma^j}\,\alpha_j\,{{\mathcal H}^{N-1-j}{\res{\Sigma_j}}}}
\label{eq:asymptotic_jump_term}
\end{equation}
  where 
  \begin{itemize}
 \item $\alpha_j$ are $\bR$-valued density functions;
 \item $\Sigma_j$ are $(N-1-j)$-dimensional sets;
 \item the discontinuity set of $n$ satisfies $\Sigma = S_{\Omega} =\displaystyle\bigcup_j\Sigma_j$
\end{itemize}
\end{conj}

\subsubsection{Case of a round sphere in $\bR^N$}
In this case, the associated skeleton $\Sigma$ is a singleton. Without loss of generality, we can assume that $\Sigma = \{ 0_{\bR^N} \}$ and therefore we have
$$
 n (x) =  - \frac{x}{|x|}, \, \forall x \neq 0.
$$
Thanks to the definition of $ n^\sigma$ and its symmetric properties, we get that
\begin{equation}
 n^\sigma(x) = w\left( \frac{|x|}{\sigma} \right)  n(x), \, \forall x \neq 0,
\label{eq:n_sigma_point}
\end{equation}
where $w:\bR^+ \to \bR$ is a smooth function such that $w(0) =0$, $\lim_{x \to +\infty} w(x) = 1$ and 
$$
0 \leq w(x) \leq 1 \text{ for every } x\geq0 .
$$
Since $\sum_{i, j} n_i \partial_j n_i n_j =0$ a.e., we get from \eqref{eq:n_sigma_point} that
$$
\begin{aligned}
S n^\sigma(x)&=  \sum_{i, j} w\left( \frac{|x|}{\sigma} \right)  n_i \left( w\left( \frac{|x|}{\sigma} \right) \partial_j  n_j \right) w \left( \frac{|x|}{\sigma}\right ) n_j 
\\
&=  - \sum_j 
\left( w \left( \frac{|x|}{\sigma} \right) \right)^2 w'\left( \frac{|x|}{\sigma} \right)  n_j \frac{x_j}{\sigma |x|}
\\
&=  - \frac{1}{\sigma}\left( w \left( \frac{|x|}{\sigma} \right) \right)^2 w'\left( \frac{|x|}{\sigma} \right) 
\end{aligned}
$$
Therefore, for every $\varphi \in C^\infty_c (\bR^N)$, we get that
$$
\begin{aligned}
\langle S n^\sigma, \varphi \rangle &=
 - \frac{1}{\sigma}
\int_{\bR^N}   \left( w \left( \frac{|x|}{\sigma} \right) \right)^2 w'\left( \frac{|x|}{\sigma} \right) \varphi (x)\,  dx
\\
&=
 - \frac{1}{\sigma}
\int_0^\infty \int_{\bS_r} \left( h \left( \frac{r}{\sigma} \right) \right)^2 w'\left( \frac{r}{\sigma} \right) \varphi(r, s) ds_{\bS_r} dr
\\
&=  - \frac{1}{\sigma} \int_0^\infty \left( w \left( \frac{r}{\sigma} \right) \right)^2 w'\left( \frac{r}{\sigma} \right) 
r^{N-1} \left (\int_{\bS^{N-1}} \varphi(r, s)\,  ds\right ) \, dr
\\
&=  - \sigma^{N-1}  \int_0^\infty w(u)^2 w'(u) \left (\int_{\bS^{N-1}} \varphi(\sigma u, s) \,  ds \right )
u^{N-1} \, du
\\
&\sim 
 - \sigma^{N-1} \int_{\bR^N} w(|x|)^2 w'(|x|) dx \varphi(0)
,
\end{aligned} 
$$
which gives 
\begin{equation}
\frac{1}{\sigma^{N-1}} \langle S  n^\sigma, \varphi \rangle 
\to \left ( - \int_{\bR^N} h(|x|)^2 h'(|x|) dx \right ) \varphi(0)
,
\label{eq:jump_term_points}
\end{equation}
as $\sigma \to 0$.

\subsubsection{Case of a circular torus in \texorpdfstring{$\bR^3$}{R3}}
 We now discuss the case were the skeleton $\Sigma = \{ (x_1, x_2 ,0 ) \in \bR^3 \, | \, x_1^2 + x_2^2 = 2^2\}$ is a planar circle in $\bR^3$, 
 and $ \Omega_\sigma= \{ x \in \bR^3\, | \, d(x, \Sigma) \leq \sigma \}$ is a circular torus obtained by thickening the curve.

For $\sigma>0$ small enough, and for every $x \in \Omega_\sigma$, we can write that 
$$x = s + \sigma u,$$ where $s = \pi_\Sigma(x)$ is the orthogonal projection of $x$ on $\Sigma$ and $u\in N_s \Sigma$ is a unit vector of the normal space at $s$. In particular, the map
$$
\begin{aligned}
  \tau_\sigma : N &\to \Omega_\sigma
  \\
  (s, u) &\mapsto s + \sigma u,
\end{aligned}
$$
where $N = \bigsqcup_{s \in \Sigma}  N_s =  \bigsqcup_{s \in \sigma} \{ u \in N_s \Sigma \, | \, |u | =1 \} $,  is a diffeomorphism.
We take the canonical orthonormal basis on $\Omega_\sigma$ denoted by $(e_s, e_{u_1}, e_{u_2})$. Notice that this remains true for any regular curve in $\bR^3$.

Thanks to the definition of $\tau_\sigma$, for every $x_0^\sigma \in \Omega_\sigma$, we can write that 
$$ 
x_0^\sigma = \tau_\sigma (s_0, u_0)
.
$$
Using integration by substitution, we obtain that
\begin{equation}
\begin{aligned}
 n^\sigma(x_0^\sigma) &= \int_{\Omega_\sigma}  n(y) h^\sigma(x_0^\sigma - y) \, dy
\\
                      &= \int_\Sigma \left ( \int_{N_s} n(\tau_\sigma (s, u)) h^\sigma(\tau_\sigma(s_0, u_0) - \tau_\sigma(s, u))J_{\tau_\sigma}(s, u)\, du\right )   \, ds 
                      \\
                      &= \frac{1}{\sigma^3}\int_\Sigma \left ( \int_{N_s} n(\tau_\sigma (s, u)) h\left(\frac{\tau_\sigma(s_0, u_0) - \tau_\sigma(s, u)}{\sigma} \right) J_{\tau_\sigma}(s, u)\, du \right )  \, ds 
                      \\
                      &= \frac{1}{\sigma^3}\int_\Sigma \left ( \int_{N_s} n(\tau_\sigma (s, u)) h\left(\frac{s_0 - s}{\sigma} + (u_0 - u) \right) J_{\tau_\sigma}(s, u) \, du\right )  \, ds,
\end{aligned}
\label{eq:jump_term_filaments_proof1}
\end{equation}
where $J_{\tau_\sigma}(s, u)$ is the Jacobian of $\tau_\sigma$.
Since $J_{\tau_\sigma}(s, u) = \sigma^2 J_{\tau_1}(s, u) = O(\sigma^2)$, thanks to \eqref{eq:jump_term_filaments_proof1} and the symmetries of $n$ and $h$, we get that 
\begin{equation}
  \partial_u n^\sigma (x_0^\sigma) = O \left ( \frac{1}{\sigma} \right) \mbox{ and }
  \partial_s n^\sigma(x_0^\sigma) = 0,
  \label{eq:jump_term_filaments_proof2}
\end{equation}
as $\sigma \to 0$.
Finally, injecting \eqref{eq:jump_term_filaments_proof2} into Definition \ref{def:jump_term} of $S n^\sigma$, we have that  
\begin{equation}
\begin{aligned}
\langle S   n^\sigma, \varphi \rangle &= 
\sum_{i, j} \int_{\Omega_\sigma}  n_i^\sigma \partial_j  n_i^\sigma  n_j^\sigma \, \varphi \, dx
\\ 
                                      &= O\left (\int_\Sigma \left ( \int_{N_s} \partial_u n^\sigma(\tau_\sigma(s, u)) \varphi (\tau_\sigma(s, u)) J_{\tau_\sigma}(s, u)\, du \right) \, ds \right) 
                                      \\ 
                                      &= O\left( \frac{\sigma^2}{\sigma}\int_\Sigma \varphi(s) \, ds\right )
                                      \\
                                      &= O\left( \sigma \int_\Sigma \varphi(s) \, ds\right ),
\end{aligned}
\end{equation}
as $\sigma \to 0$.

\begin{rmk}

The numerical computations in Figures \ref{Jump2D} and~\ref{Jump3D} show that the magnitude of the skeletal term appears to be linear with respect to $\sigma$ and does not depend on the size of the skeleton. 
For example, in the case of two concentric circles as in Figure~\ref{Jump2D}, we can see that the skeletal term calculated at the centre has the same order of magnitude (up to a sign difference) as on the circle. This does not contradict \eqref{eq:asymptotic_jump_term} for the following reason: on the one hand, the Dirac mass $\delta_0$ at the centre can be represented numerically by the inverse of the pixel area, which gives in 2D

$$
\mu_0(\Sigma) \simeq \frac{1}{\sigma^2}.
$$
On the other hand, the uniform measure on the discrete circle 
satisfies $\mu_1(\sigma)  \simeq \frac{1}{L \sigma } 
$ since the circle of length $L$ is approximated numerically by a band of area $L \sigma$ in $2D$.
Therefore, thanks to \eqref{eq:asymptotic_jump_term}, we have that $
\sigma \mu_0(\Sigma) \simeq L\mu_1(\sigma) \simeq \frac{1}{\sigma},
$ which justifies the fact that the skeletal term has approximately the same magnitude at the center and on the circle in Figure \ref{Jump2D}.
\end{rmk}

\section{ Allen--Cahn equation with a  skeleton-aware  forcing term} 
In this section, we address the approximation of the mean curvature flow with a forcing term based on skeletal approximation that acts as a thickness constraint for the evolving set. Thanks to the results on the characterisation of singular sets by the skeletal term of Definition \ref{def:jump_term}, as well as our analysis of the perturbed Allen-Cahn equation, a natural candidate for the phase field model we are trying to identify is the Allen-Cahn equation coupled with  the skeletal term, i.e.,
\begin{equation}
\partial_t u^\varepsilon = \Delta u^\varepsilon - \frac{W'(u^\varepsilon)}{\varepsilon^2} (1 + f^{\sigma}_{u^{\varepsilon}})
\label{eq:allen-cahn_jump_term}
\end{equation}
where 
$$ f^{\sigma}_{u^{\varepsilon}} = c (h^\sigma * |S  n^\sigma_{u^{\varepsilon}}|) \quad \text{and } \quad n^\sigma_{u^{\varepsilon}} = h^\sigma * \frac{\nabla u^{\varepsilon}}{|\nabla u^{\varepsilon}|}.$$
Here, $h^\sigma$ is again the Gaussian kernel of standard deviation $\sigma>0$  and $c$ is a constant to be chosen sufficiently large in order to enforce 
the thickness constraint. 

The theoretical analysis of such a phase field model seems rather difficult and is beyond the scope of this article.  Our first goal is to propose a semi-implicit numerical discretisation 
of this model and to illustrate its potential for handling minimum thickness constraint and, more generally, for obtaining approximations of motion by mean curvature of filaments and other fine structures.  

 The code used for these experiments is  available on the GitHub repository. \footnote{\url{https://github.com/eliebretin/skeleton_Allen_Cahn.git}}

\subsection{ Numerical discretization and a first example }

To approximate numerically the above flow, we propose a quasi-static approach with explicit integration of the skeletal term. More precisely, we consider the sequence of approximate solutions $(u^n)_n$  defined recursively 
by $u^{n+1}(x) \simeq v(x,\delta_t)$ where $v$ is solution to the following PDE 
\begin{equation}
\begin{cases}
   \partial_t v(x,t) &= \Delta v(x,t) - \frac{W'(v(x,t))}{\varepsilon^2}(1 + f^{\sigma}_{u^{n}}(x)) \\
   v(\cdot,0) &= u^{n}
  \end{cases}
\label{eq:quasi-static}
\end{equation}
and $\delta_t$ represents the time step. Notice that the above equation is variational as it is the $L^2$-gradient flow of the perturbed Cahn--Hilliard energy 
$$ J^{\varepsilon}_{u^n}(v) =  \int_{\bR^N}(\frac{\varepsilon}{2} |\nabla v|^2 + \frac{1}{\varepsilon} W(v) \left( 1 + f^{\sigma}_{u^{n}}) \right)dx,$$

From a numerical point of view, equation~\eqref{eq:quasi-static} is solved  in a box $Q$ with  periodic boundary conditions. We use 
a semi-implicit approach based on a convex-concave splitting of the perturbed Cahn--Hilliard energy. More precisely, we define the solution $u^{n}$ as an approximation of $u$ at the time $n \delta_t$ given by the following scheme
$$ \frac{u^{n+1} - u^{n}}{\delta_t} =  (\Delta u^{n+1} - \frac{\alpha}{\epsilon^2} u^{n+1})    -  \left((1 + f^{\sigma}_{u^{n}})  \frac{ W'(u^n)}{\varepsilon^2} -  \frac{\alpha}{\epsilon^2} u^{n} \right). $$
 The stabilization coefficient $\alpha$ is  assumed to be sufficiently large  to ensure the concavity of the functional 
$$v \mapsto \int_{Q} (1 + f^{\sigma}_{u^{n}}) \frac{W(v)}{\epsilon^2}  - \alpha \frac{v^2}{2\varepsilon^2} dx.$$
Indeed, in that case, the  perturbed Cahn--Hilliard energy  satisfies the decreasing property
$$ J^{\varepsilon}_{u^n}(u^{n+1}) \leq J^{\varepsilon}_{u^n}(u^{n}),$$
and the scheme is unconditionally stable with respect to energy.  A convenient formulation of the scheme is 
$$
 u^{n+1}=   ( I_d - \delta_t ( \Delta  - \frac{\alpha}{\epsilon^2} I_d) )^{-1} \left(u^{n}  - \delta_t \left( (1 + f^{\sigma}_{u^{n}}(x)) \frac{W'(u^n)}{\varepsilon^2} -  \frac{\alpha}{\epsilon^2} u^{n} \right)\right) $$
because the operator $( I_d - \delta_t ( \Delta  - \frac{\alpha}{\epsilon^2} I_d) )^{-1} $ can be computed very efficiently in Fourier space using 
the Fast Fourier Transform and the periodic boundary conditions on the box $Q$. \\

In practice, we consider the box  $Q = [-0.5,0.5]^3$, discretized with $N$ nodes on each axis.  All the $3D$ numerical experiments 
presented in this section are  obtained  using a resolution of $N = 2^7$, a phase-field parameter $\varepsilon = 2/N$, 
and $\delta_t = \varepsilon^2$. As for the skeletal term, we set $\sigma^2 = 0.1 \varepsilon^2$ and $c = 0.35 \varepsilon N^3$. \\

The first numerical experiment illustrates the influence of the skeletal term in our model. We consider a dumbbell as the initial set, and we recall that the classical mean curvature flow produces a topology change in finite time.
This is illustrated on the first line of Figure \ref{fig_dumbbell} where we plot the $0$ level set of an approximate solution $u^{n}$ of the Allen-Cahn equation computed at different times $t^{n} = n h$ without using the skeletal term (which corresponds to using $c=0$).
On the second line, we plot the level set $0$ of the approximate solution $u^{n}$ of our model.. We can clearly see here that the skeleton term forces the topological conservation along iterations.

\begin{figure}[htbp]
\centering
		\includegraphics[width=15cm]{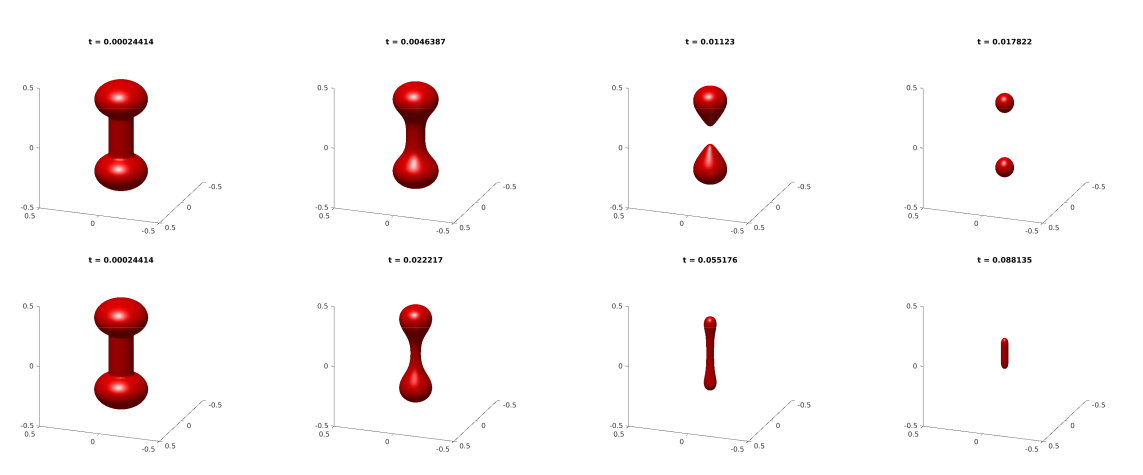} 
\caption{Evolution of a dumbbell: the solution $u^{n}$ plotted at different times $t$.
The first and second lines correspond to the approximate mean curvature flow with or without, respectively, the additional skeletal term.}
\label{fig_dumbbell}
\end{figure}

\subsection{Second example: a circle in 3D}
We consider the evolution of a circle in $3D$ as an example of a smooth codimension 2 mean curvature flow. The advantage is that the evolution is explicit and corresponds to the evolution of a circle whose radius satisfies

$$
V_n = -\frac{d}{dt} R = \frac{1}{2R},
$$
thus
$$
R(t) = \sqrt{ R(0)^2 - 2t}
$$

We  plot in Figure \ref{filament_circle} the solution at
different times $t^{n} = n h$. We can clearly observe a circle which decreases along the iterations.
We also illustrate in the last plot that the evolution along the iterations of the squared mass of $u$,
 $$ t \mapsto \left(\int_Q u(x,t)\, dx \right )^2,$$ 
is proportional to the square of the radius 
of the circle  which is expected to decrease linearly in the case of the mean curvature flow. 

 We can clearly  observe such a numerical decrease (except at the very beginning of the simulation), which means that the circle evolves consistently with the higher codimension mean curvature flow. One of the reasons for the 'bad' behaviour at the very beginning is probably an inconsistent initialisation that is quickly regularised by the flow.
 
\begin{figure}[htbp]
\centering
		\includegraphics[width=16cm]{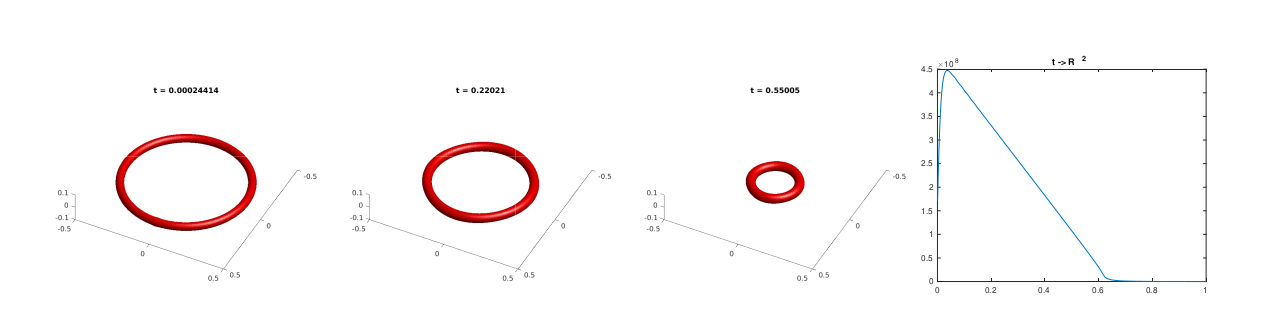} 
\caption{Evolution of a circle: the solution $u^{n}$ plotted at different times $t$ 
and the squared mass $t \mapsto (\int_Q u(x,t) dx)^2$ along the iterations.}
\label{filament_circle}
\end{figure}

\subsection{ Influence of the parameter $\sigma$  on the mean curvature flow approximation}
To illustrate the influence of the thickening parameter $\sigma$ for the approximation of the higher codimension mean curvature flow, at least in the case of a circle of radius $R$ of codimension $2$, we propose to compute the average of the mean curvature vector on a slice of the torus of thickness $\sigma$  to see that this estimate leads to a normal velocity given by  
$$ V_n =  \frac{1}{2 R} + O(\delta^2/R^3).$$
In particular, as long as $\delta^2/R^2 << 1$, the thickening technique provides a good approximation of the codimension $2$ mean curvature flow.\\

Consider, for all $u \in \bS$, $\Psi(u) = 
\left (
\begin{aligned}
&R \cos u 
\\
&R \sin u 
\end{aligned}
\right )
,
$ 
which parametrizes a circle $\Gamma$ of radius $R>0$ lying on the horizontal plane $\bR^2 \times \{0\}$ in $\bR^3$.
The curvature of $\Gamma$ is $1/R$.
\par
We start by thickening the curve with an offset $\delta>0$, hence we obtain a circular torus.
Let $\Phi$ be a parameterization of the torus defined by 
$$
\Phi(u, v) =  
\left( 
\begin{aligned}
(R + \delta \cos v) &\cos u 
\\
(R + \delta \cos v) &\sin u 
\\
\delta &\sin v
\end{aligned}
 \right).
$$
By direct computations, the inner normal $ n$ and the scalar mean curvature $H$ can be written as 
$$
 n  (u, v)
=
-
\left( 
\begin{aligned}
&\cos u \cos v 
\\
&\sin u \cos v
\\
&\sin v
\end{aligned}
 \right)
$$
and 
$$
H(u, v)  = - \left( \frac{1}{\delta} + \frac{\cos v}{ R + \delta \cos v}\right)
.
$$
The velocity of the torus on $\Phi(u, v)$ under the flow \eqref{eq:forced_mean_curvature_flow} is 
\begin{equation}
{\bf V}_{\Phi(u, v)} = {\bf H}(u, v) + {\bf g}_\sigma (\delta)
.
\label{eq:torus_velocity}
\end{equation}
We  denote the mean curvature vector as ${\bf H}:= H  n$.
For every $u \in \bS$, the motion of the point $\Psi(u)$ is the average velocity of the torus on the associated sectional points, that is the mean value of $V_{\Phi (u, v)}$ for all $v \in \bS$.
\begin{equation}
\begin{aligned}
{\bf  V_{\Psi(u)}} &= 
\frac{1}{2\pi} \int_0^{2\pi} 
V_{\Psi(u, v)}  n (u, v) dv 
\\
&=
\frac{1}{2\pi} \int_0^{2\pi} {\bf  H}(u, v) dv + 
\frac{1}{2\pi} \int_0^{2\pi} g_\sigma(\delta) n(u, v) dv 
\\
&= 
 - \frac{1}{2\pi}
\begin{pmatrix}
&\cos u \int_0^{2\pi} \frac{\cos^2 v}{R + \delta \cos v} dv 
\\
&\sin u \int_0^{2\pi} \frac{\cos^2 v}{R + \delta \cos v} dv 
\\ 
&\int_0^{2\pi} \frac{\cos v \sin v}{ R + \delta \cos v} dv 
\end{pmatrix}
.
\end{aligned}
\label{eq:torus_velocity1}
\end{equation}
Notice that the forcing term $g_\sigma(\delta)$ in \eqref{eq:torus_velocity} does not have in average any impact on the velocity in \eqref{eq:torus_velocity1}.
\par
Remark that $\displaystyle
\int_0^{2\pi} \frac{\cos v \sin v}{ R + \delta \cos v} dv 
= 0
$
and 
\begin{equation}
\begin{aligned}
\int_0^{2\pi} \frac{\cos^2 v}{R + \delta \cos v} dv 
&=
\int_0^{\frac{\pi}{2}} \left ( \frac{\cos^2 v}{R + \delta \cos v}  + \frac{\cos^2 v}{R - \delta \cos v} \right ) dv
+
\int_{-\frac{\pi}{2}}^0 \left ( \frac{\cos^2 v}{R + \delta \cos v}  + \frac{\cos^2 v}{R - \delta \cos v} \right ) dv
\\
&=
4R \int_0^{\frac{\pi}{2}} \frac{\cos^2 v}{R^2 - \delta^2\cos^2 v}dv
\\
&=
\frac{4R}{\delta^2} \left(  R^2 \int_0^{\frac{\pi}{2}}  \frac{1}{R^2 - \delta^2 \cos^2 v} dv - \frac{\pi}{2}\right)
.
\end{aligned}
\label{eq:torus_velocity2}
\end{equation}
Moreover, by the change of variables $x=\tan v$, we get that
\begin{equation}
\int_0^{\frac{\pi}{2}} \left (\frac{1}{R^2 - \delta^2\cos^2 v} \right )dv = \int_0^\infty \left ( \frac{1}{R^2(1+x^2) - \delta^2} \right ) dx
= \frac{\pi}{2} \frac{1}{R \sqrt{R^2 - \delta^2}}.
\label{eq:torus_velocity3}
\end{equation}
Therefore, we deduce from \eqref{eq:torus_velocity1}, 
\eqref{eq:torus_velocity2} and \eqref{eq:torus_velocity3} that
\begin{equation}
{\bf  V}_{\Psi (u)} =  -
\frac{R}{\delta^2} \left( 1 - 
\frac{1}{\sqrt{1 - \frac{\delta^2}{R^2}}} \right)
\begin{pmatrix}
\cos u
\\
\sin u
\\ 
0
\end{pmatrix}
\sim 
\frac{1}{2R} \left ( 1+ O\left (\frac{\delta^2}{R^2}\right ) \right )  n_{\Psi(u)},
\label{eq:torus_asym_velocity}
\end{equation}
where $ n_{\Psi(u)}$ is the normal vector of $\Psi$.
Consequently, as explained previously, the circle of radius $R>0$ moves with a normal velocity $V_n \sim \frac{1}{2R}$ at first order, i.e. consistently with the codimension 2 mean curvature flow.

\subsection{Applications to smooth or non smooth curves (and filaments) in 3D}

We present two filament evolutions (i.e. smooth curves) in figure \ref{fig_filament1}.
The first column shows the evolution of a simple filament which converges, using periodic boundary conditions, to a simple line. The second column shows the evolution of a more complex filament with the presence of a triple junction. 

\begin{figure}[!htbp] 
\centering
    	\includegraphics[width=16cm]{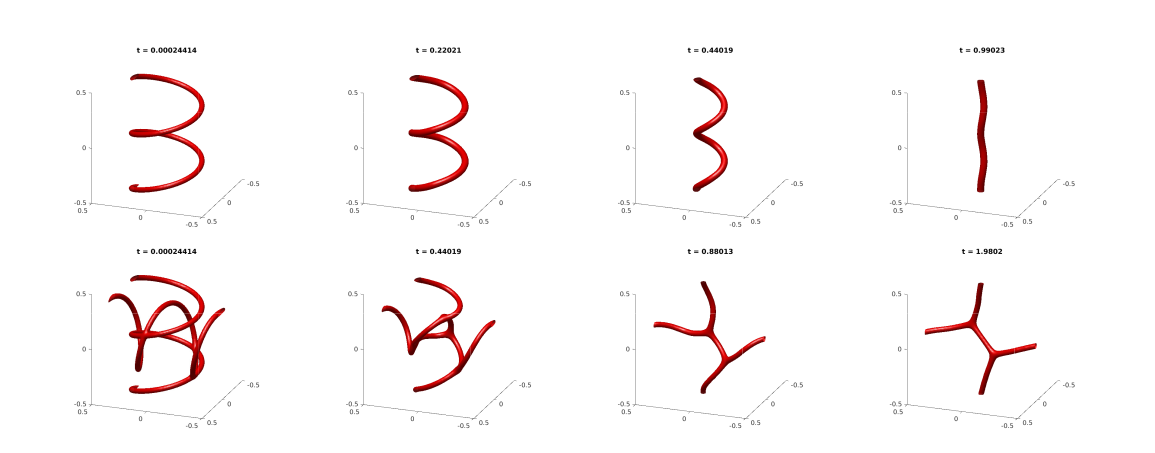} 
        \caption{Numerical approximations of mean curvature flows of filaments: each line corresponds to the evolution obtained at different times from a given initial condition.}
\label{fig_filament1}
\end{figure}

To get a better idea of the numerical flows calculated with our model, videos (.avi format) can be downloaded from the following addresses.

\begin{itemize}
\item Video of the flow shown in Figure~\ref{fig_filament1}, top: 
\href{http://math.univ-lyon1.fr/homes-www/huang/manuscrit/video/filament1.avi}{Filament~1}.
\item Video of the flow shown in Figure~\ref{fig_filament1}, bottom: 
\href{http://math.univ-lyon1.fr/homes-www/huang/manuscrit/video/filament2.avi}{Filament~2}.
 \item In the previous video, the filament remains connected but the inner cycle disappears. It is actually possible to play with that by simply changing the weight of the skeletal term: a higher weight yields a higher sensitivity. This is the choice made for the numerical simulation shown in the following third video where the inner cycle is conserved: \href{http://math.univ-lyon1.fr/homes-www/huang/manuscrit/video/filament3.avi}{Filament~3}. 
\end{itemize}

\section{Application to the Steiner problem}
The numerical experiments and observations of the previous sections on the Allen-Cahn equation coupled with 
the skeletal term \eqref{eq:allen-cahn_jump_term} lead us naturally to address the application
to the Steiner problem in 3D.  Recall that the Steiner problem consists in finding, for a given  collection of points 
$a_0, \dots, a_N \in Q$, a compact connected set $K \subset Q$ containing all 
the $a_i$'s and  having minimal length. In other words, it amounts to solving the following minimization problem 
\begin{equation}
 \min \{  \cH^1(K),\, K \subset Q, \,K \;\text{connected,}\;a_i\in K, \forall i\}.
\end{equation}
The numerical approximation of solutions to this problem is notoriously difficult, especially in dimension $\geq 3$,  see for instance \cite{Karp,GeoSteiner,DAmbrosio,MR3337998,BonafiniOrlandiOudet,BonafiniOudet,ChambolleFerrariMerlet2019-1,ChambolleFerrariMerlet2019-2,MR4011534}. The model we propose provides an effective and natural way to approximate these solutions in 3D (and actually even in higher dimension for the extension is straightforward).

\subsection{A $\sigma$-thickened Steiner problem}
We consider an approximation of the Steiner problem in 3D for the $\sigma$-tubular set  $K_{\sigma}:= \{ x \in \bR^3\, |\, \operatorname{dist}(x, K) \leq  \sigma \}$ which is the $ \sigma$-thickening of $K$. 
We choose the thickening parameter $\sigma$ to be small enough as in Lemma~\ref{claim:minimal_distance} so that  
the minimal distance between the boundary of $K_\sigma$ and its skeleton $K$ is ensured.
\par
\noindent It is clear that the length of $K$ is approximated by the perimeter of $K_{\sigma}$ in the sense that 
$$\cP(K_{ \sigma}) \simeq 2\pi \sigma \cH^1(K).$$ 
Moreover, the property that $K$ contains all points $a_i$ can be replaced by the \emph{inclusion constraint}:
\begin{equation}
  \cup_{i=1}^{N}B(a_i, \sigma) \subset  K_{ \sigma},
\label{eq:inclusion_constraint}
\end{equation}
 which means that $K_{\sigma}$ contains all the balls $B(a_i, \sigma)$ of radius $\sigma$. 

We now explain how to obtain  a phase-field approximation of the   minimization problem
\begin{equation*}
 \min \{  \cH^2(K_{ \sigma}) \, |\,     \cup_{i=1}^{N}B(a_i, \sigma) \subset K_{ \sigma} \text{ and } \operatorname{dist}(\partial K_{ {\sigma}}, \Sigma(K_{{\sigma}})) \geq {\sigma}  \},
\end{equation*}
where $\Sigma(K_{{\sigma}})$ is the skeleton of $\partial K_{{\sigma}}$.

\subsection{Phase-field approximation of the thickened Steiner problem}

There is a natural way to approximate the inclusion constraint \eqref{eq:inclusion_constraint} using phase fields, see~\cite{bretin_slice} . This can be done by first introducing the function $u^{\varepsilon}_{\text{in}}$ defined by
$ u^{\varepsilon}_{\text{in}}(x) := q \left( \frac{\operatorname{dist}(x,\cup_{i=1}^{N} \{a_i\}) - \tilde{\sigma}}{\varepsilon} \right)$,
where  $\operatorname{dist}(x,\cup_{i=1}^{N} \{a_i\})$  denotes the distance function to the points $a_i$ and
$q$ is the usual optimal profile associated with the double-well function. The inclusion constraint \eqref{eq:inclusion_constraint} can thus be implemented by considering the inequality constraint:  
$u^{\varepsilon}_{\text{in}} \leq u^\varepsilon$.
Indeed, it is not difficult to observe that 
$$  u^{\varepsilon}_{\text{in}} \leq u^\varepsilon \quad  \implies  \quad  \cup_{i=1}^{N}B(a_i,\tilde{\sigma}) \subset \left\{ x ; u^\varepsilon(x) \geq 0  \right\} $$
Therefore, we  define recursively the sequence $\{ u^{n} \}_{n \in \bN}$ by 

$
u^{n+1}(x)= \max(\tilde{u}^{n+1},u^{\varepsilon}_{\text{in}}(x)),
$
where
$$
 \tilde{u}^{n+1}=   ( I_d - \delta_t ( \Delta  - \frac{\alpha}{\epsilon^2} I_d) )^{-1} \left(u^{n}  - \delta_t \left( (1 + f^{\sigma}_{u^{n}}(x)) \frac{W'(u^n)}{\varepsilon^2} -  \frac{\alpha}{\epsilon^2} u^{n} \right)\right) $$

\subsection{Numerical experiments}
As previously, we define $Q = [-0.5,0.5]^3$ and use the following settings: $N = 2^7$, $\varepsilon = 2/N$, $h = \delta_t = \varepsilon^2$,
$\sigma^2 = 0.1 \varepsilon^2$,  $c = 0.35 \varepsilon N^3$, and $\tilde{\sigma} = 0.02$. 

The first example in Figure \ref{fig_steiner_cube} represents the case where the $a_i$'s are the vertices of a cube. The initial set  
is also a cube containing all  nodes $a_i$. We show in Figure~\ref{fig_steiner_cube} approximate solutions at different times of the phase field flow.
\par

We note that although the initial set is of dimension $3$ (the initial cube), the stationary solution is close to a tubular thickening of a set of dimension $1$ containing all the nodes $a_i$. 
We observe that this solution is very close to the Steiner tree associated with the vertices of the cube, in particular all the angles at each triple junction are equal to $2 \pi/3$.
\par
As a comparison, we  consider in Figure \ref{fig_steiner_alea} an example with $10$ points randomly distributed in $Q$ which  leads exactly to the same conclusion. 

 \begin{figure}[H]
 \centering
 	\includegraphics[width=16cm]{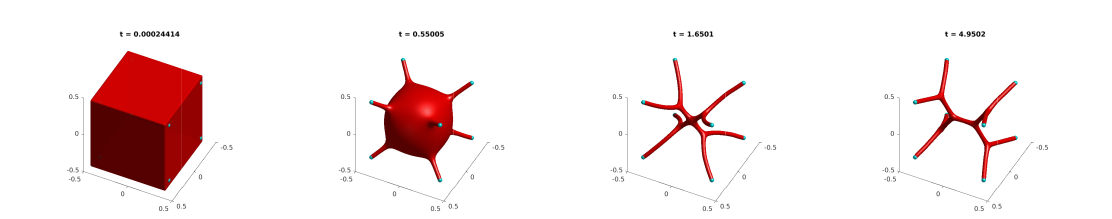} 
\caption{A Steiner tree associated with the vertices of a cube: illustration of the approximate solutions at different times along the numerical flow.
 The red and blue interfaces are, respectively, the $0$-level set of $u^{n}$ and the set $ \cup_{i=1}^{N}B(a_i,\tilde{\sigma})$.}
 \label{fig_steiner_cube}
 \end{figure}
 \vspace*{-0.5cm}
  \begin{figure}[!htbp]
 \centering
         	\includegraphics[width=16cm]{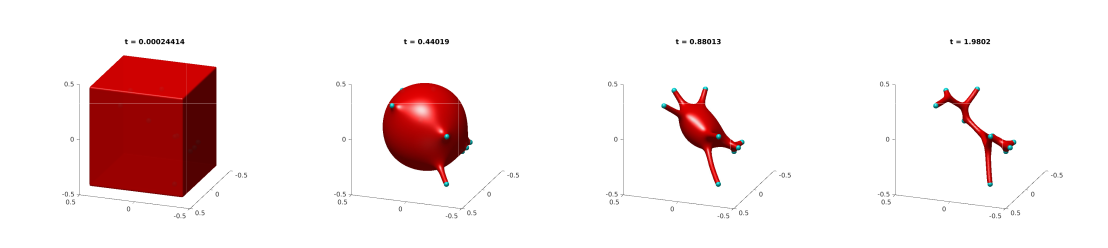} 
\caption{A Steiner tree associated with $10$ points randomly chosen in $Q$:  illustration of the approximate solutions at different times along the numerical flow. 
 The red and blue interfaces are, respectively, the $0$-level set of $u^{n}$ and the set $ \cup_{i=1}^{N}B(a_i,\tilde{\sigma})$.}
 \label{fig_steiner_alea}
 \end{figure}

To get a better idea of the numerical flows computed with our model, videos (in .avi format) can be downloaded at the following addresses:

\begin{itemize}
\item Video of the flow of Figure~\ref{fig_steiner_cube}: 
\href{http://math.univ-lyon1.fr/homes-www/huang/manuscrit/video/Steiner_cube-vertices.avi}{Approximation of the Steiner set of cube's vertices}.
\item Approximation of the Steiner tree associated with 50 randomly chosen points: 
\href{http://math.univ-lyon1.fr/homes-www/huang/manuscrit/video/Steiner_50-random-points.avi}{Approximation of the Steiner set of 50 random points}.
\end{itemize} 

\section{Application to the Plateau problem}

Our last numerical application is devoted to the celebrated Plateau problem.
Recall that the Plateau problem was formulated by Lagrange in 1760 and consists in finding, for a given closed Jordan curve $\gamma$, a surface $E$ in $\bR^3$ with a (locally minimal) area such that 
the boundary of $E$ coincides with  $\gamma$.    
In other words, it amounts to solving the following minimization problem:

\begin{equation}\label{Intro:PlateauPb}
 \min \{  \cH^2(E),\; E \subset \Omega,\; \text{connected and such that } \partial E = \gamma \}.\end{equation}

The existence and regularity of solutions to this problem have been studied in different contexts. For example, in the case of smooth and orientable solutions, the reference \cite{MR1501590} provides relevant insights, while \cite{MR117614} deals with the existence and uniqueness of soap films, treating both orientable and non-orientable surfaces, and potentially multiple junctions. \\

Many numerical methods have been developed to approximate minimal surfaces. In \cite{MR1502829,MR0458923,MR942778,MR1613695,MR1613699}, a parametric representation of the surface is used, itself discretised by a finite element approach.  Numerical techniques involving an implicit representation of minimal surfaces have been developed for instance in \cite{MR1214016}, \cite{MR2143330}, where a level set method is employed. Phase field approaches using the theory of currents have been introduced and analysed, for example in \cite{ChambolleFerrariMerlet2019-2} or \cite{Wang:2021:CMS}, offering particularly suitable numerical approximations, but restricted to oriented surfaces.

\subsection{A  $\sigma$-thickened Plateau problem}

We consider a $ \sigma$-thickened minimal surface problem for the $\sigma$-tubular thickening of a given Jordan curve $\gamma$  
 $$ \gamma_{{\sigma}} := \left\{ x ; \operatorname{d}(x,\gamma) \leq {\sigma} \right\},$$
 by considering  the following  minimization problem:
\begin{equation}\label{Intro:PlateauPb_sigma}
 \min \left \{  \cP(E_{{\sigma}}) + \frac{c}{{\sigma}} \cH^3 (E_{{\sigma}}) ; E_{{\sigma}} \subset \bR^N, \text{connected and } \gamma_{{\sigma}} \subset E_{{\sigma}} \right \}.
 \end{equation}
where $\cH^3(E_{{\sigma}})$ stands for the volume of $E_{{\sigma}}$.
 Here, ${\sigma}$ being chosen sufficiently small, the volume term is present to ensure  that 
 $E_{{\sigma}}$ has a  thickness of size ${\sigma}$ which requires the existence of a connected  set $E$ 
 such that 
 $$  E_{{\sigma}} \simeq \{ x \in \bR^N ; \operatorname{dist}(x,E) \leq {\sigma} \} \quad  \text{ and } \quad  \cP(E_{{\sigma}}) \simeq 2 \cH^2(E).$$
Figure~\ref{fig_plateau_truc_volume} shows what may happen if this volume term is not used: it it not possible to get in the limit a "thin" volume which approximates a minimal surface. 

As before, from an initial connected set, the connectedness property is preserved thanks to the skeletal term in the penalized mean curvature flow. 

\begin{rmk}
Unlike the thickened Plateau problem, there is no need to penalise volume in the thickened Steiner problem. The reason for this is simple: in the Steiner problem, the constraints on the vertices, combined with a property of minimising the surface area of the flow, lead to a natural reduction in volume to achieve a final tubular set.
\end{rmk}

\subsection{
Phase-field approximation of the thickened Plateau problem
}
Similarly to the thickened Steiner problem, let us define $u^{\varepsilon}_{\text{in}}$ as   
$$ u^{\varepsilon}_{\text{in}}(x) = q \left( \frac{\operatorname{dist}(x, \gamma_{\tilde{\sigma}})}{\varepsilon} \right),$$
and consider the sequence $(u^n)_n$ by combining a perturbed Allen--Cahn equation with an additional  inclusion constraint, i.e. we compute:
\begin{equation}
u^{n+1}(x)= \max(\tilde{u}^{n+1},u^{\varepsilon}_{\text{in}}),
\label{eq:phase-field_plateau-cond}
\end{equation}
where $\tilde{u}^{n+1}$ is an approximation of $v$ at time $\delta_t$ where $v$ is solution of the  flow
\begin{equation}
\left \{ 
\begin{aligned}
   \partial_t v(x,t) &= \Delta v(x,t) - \frac{W'(v(x,t))}{\varepsilon^2}(1 + f^{\sigma}_{u^{n}}(x)) + 
   \frac{c_{\text{volume}}}{\varepsilon \sigma} \sqrt{2 W(u^n)}, \\
   v(\cdot,0) &= u^{n}.
\end{aligned}
\right .
\label{eq:phase-field_plateau}
\end{equation}
More precisely, similarly to the Steiner problem, we use the following numerical approximation in practice:
$$
 \tilde{u}^{n+1}=   ( I_d - \delta_t ( \Delta  - \frac{\alpha}{\epsilon^2} I_d) )^{-1} \left(u^{n}  - \delta_t \left( (1 + f^{\sigma}_{u^{n}}(x)) \frac{W'(u^n)}{\varepsilon^2} -  \frac{\alpha}{\epsilon^2} u^{n} \right) +  \frac{c_{\text{volume}}}{\varepsilon \sigma} \sqrt{2 W(u^n)}, \right) $$

\subsection{Numerical experiments}

As previously, we define  $Q = [-0.5,0.5]^3$ and we set $N = 2^7$, $\varepsilon = 2/N$, $h = \delta_t = \varepsilon^2$,
$\sigma^2 = 0.1 \varepsilon^2$,  $c = 0.35 \varepsilon N^3$, and $\tilde{\sigma} = 0.02$. \\
In Figure~\ref{fig_plateau_truc_volume}, we plot the solution obtained at different times using the parameter $c_{\text{volume}} = 0$.
In this case, the volume is not penalized during the iteration. In particular, we can observe that 
the stationary set has not the right form in the sense that its thickness is not of size $\tilde{\sigma}$
and cannot be associated with a minimal $2$-dimensional surface. 

This experiment shows all the interest to  penalize the volume in our computations.
All following numerical experiments in this section are done with the setting $c_{\text{volume}}=1$.

The main purpose of Figure ~\ref{fig_plateau_truc} is to illustrate the influence of the initial set and its topology on the result. Each column corresponds to a particular choice of initial configuration, and the images show the numerical solution of the flow at different times. We can see  that our model is able to compute  a minimal surface associated with the given boundary.   Note, however, that the topology of this minimal surface logically depends on the choice of the initial set used.

Figure \ref{fig_plateau2} illustrates the ability of our method to approximate non-orientable minimal surfaces and minimal surfaces with triple line singularities.  The first column shows an example of a non-smooth (approximate) minimal surface, while the second gives an approximation of a Möbius strip. Note that in the latter case, we start from an initial connected set given by a cube with a cylindrical hole.

To get a better idea of numerical flows computed with our model, videos (in .avi format) can be downloaded at the following addresses.

\begin{itemize}
\item Video of the flow illustrated in Figure~\ref{fig_plateau_truc}, right:
\href{http://math.univ-lyon1.fr/homes-www/huang/manuscrit/video/Plateau1.avi}{An (approximate) Plateau solution}.
\item Approximation of a Möbius strip with a singularity line: 
\href{http://math.univ-lyon1.fr/homes-www/huang/manuscrit/video/Plateau2.avi}{A singular (approximate) Möbius strip}.
\item An hybrid solution combining a Steiner set and a minimal surface: 
\href{http://math.univ-lyon1.fr/homes-www/huang/manuscrit/video/PlateauHybrid.avi}{A Steiner-Plateau hybrid solution}.
\end{itemize}

  \begin{figure}[!htbp]
 \centering
         	\includegraphics[width=16cm]{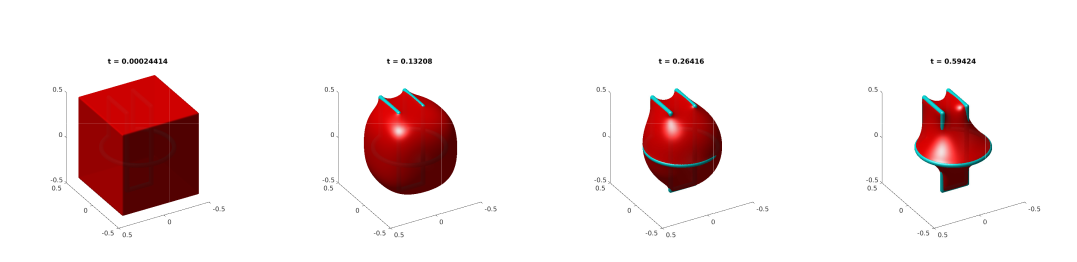} 
        \caption{Numerical illustration of the gradient flow for the thickened Plateau problem without volume penalization: the stationary shape is not at all close to a minimal surface}
         \label{fig_plateau_truc_volume}
 \end{figure}

  \begin{figure}[!htbp]
 \centering
       	 \includegraphics[width=16cm]{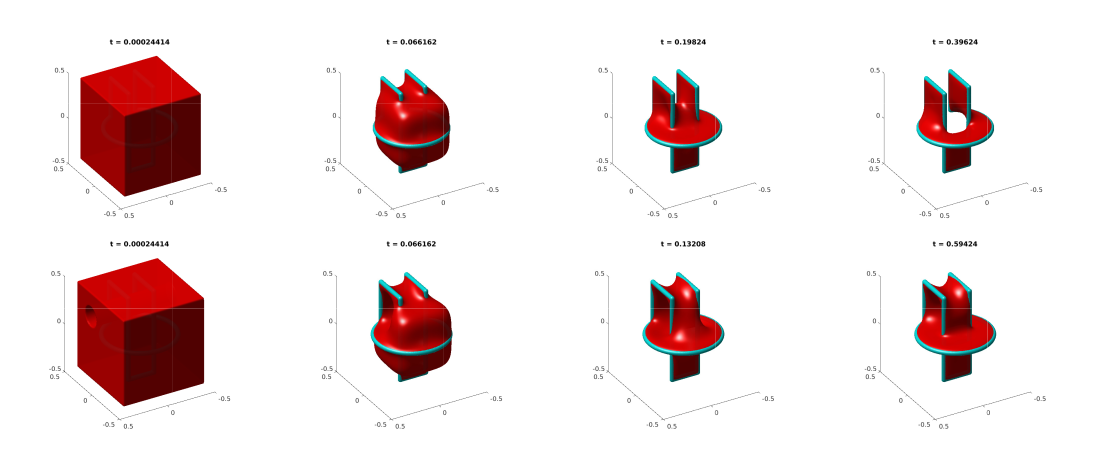} 
        \caption{Numerical approximation of solutions to the Plateau problem using our phase field method. The experiment illustrates the influence on the result of the initial set's topology. Each line shows the numerical solution at different times starting from two different initial configurations.
        }
 \label{fig_plateau_truc}
 \end{figure}
 
  \begin{figure}[!htbp]
 \centering
 	  	 \includegraphics[width=16cm]{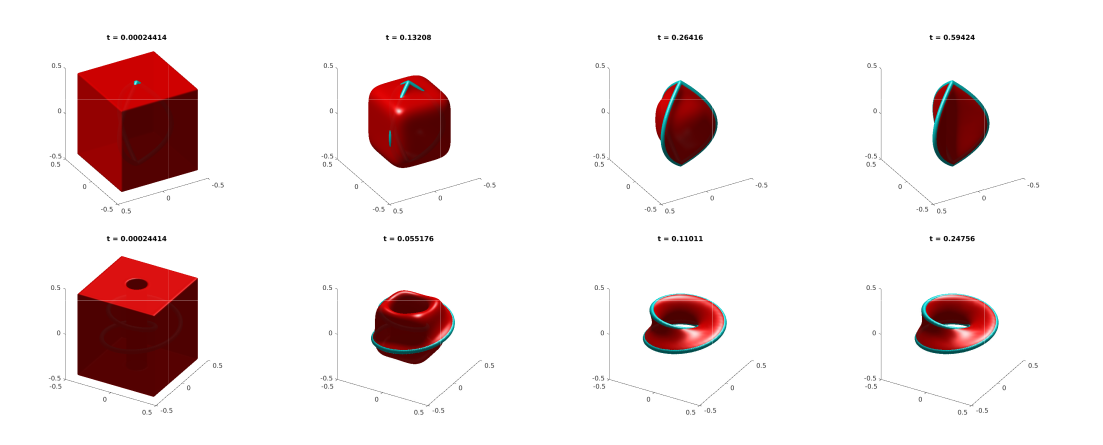} 
 	
        \caption{Numerical approximation of solutions to the Plateau problem using the proposed flow. First line: convergence to a minimal surface with a triple line singularity. 
        Second line: convergence to a M\"obius strip, an example of non-orientable surface that our model can compute.}
 \label{fig_plateau2}
 \end{figure}

\end{document}